\documentclass[11pt]{amsart}
\usepackage{amsmath,amsthm, amscd, amssymb, amsfonts, mathrsfs}
\usepackage[all]{xy}
\usepackage{color}

\usepackage[T1]{fontenc}
\allowdisplaybreaks
\newtheorem{thm}{Theorem}

\newtheorem{lema}{Lemma}
\newtheorem{prop}[lema]{Proposition}

\theoremstyle{definition}
\newtheorem{Def}[lema]{Definition}
\newtheorem{definition}[lema]{Definition}
\newtheorem{Rem}[lema]{Remark}

\usepackage{hhline}

\newcommand{\ydsnd}{{}^{\ku^{\Sn_n}}_{\ku^{\Sn_n}}\mathcal{YD}}

\newcommand{\trid}{\triangleright}
\newcommand{\link}{\sim_{\ba}}

\newcommand{\ba}{\mathbf{a}}
\newcommand{\bb}{\mathbf{b}}
\newcommand{\bt}{\mathbf{t}}

\newcommand{\fd}{finite dimensional}

\newcommand{\cA}{\mathcal{A}}

\newcommand{\cI}{\mathcal{I}}
\newcommand{\cJ}{\mathcal{J}}
\newcommand{\cO}{\mathcal{O}}

\newcommand{\Sn}{{\mathbb S}}
\newcommand{\sn}{{\mathbb S}_n}

\newcommand{\B}{{\mathbb B}}

\newcommand{\xij}[1]{x_{(#1)}}
\newcommand{\fij}[1]{f_{#1}}

\newcommand{\aij}[1]{a_{(#1)}}
\newcommand{\dij}[1]{\delta_{#1}}
\newcommand{\mij}[1]{m_{(#1)}}
\newcommand{\tij}[1]{t_{(#1)}}

\newcommand{\mdos}[2]{m_{(#1)(#2)}}
\newcommand{\mtres}[3]{m_{(#1)(#2)(#3)}}
\newcommand{\mcuatro}{m_{\textsf{top}}}

\renewcommand{\_}[1]{_{\left( #1 \right)}}

\newcommand{\ot}{{\otimes}}

\newcommand\gA{\mathfrak{A}}

\newcommand{\X}{{\mathbb X}}

\newcommand{\ku}{\Bbbk}

\newcommand{\Z}{{\mathbb Z}}

\newcommand{\C}{{\mathcal C}}

\newcommand{\toba}{{\mathcal B}}

\newcommand{\Oc}{{\mathcal O}}

\newcommand{\End}{\operatorname{End}}
\newcommand{\Aut}{\operatorname{Aut}}
\newcommand{\Ext}{\operatorname{Ext}}
\newcommand{\Ind}{\operatorname{Ind}}
\newcommand{\Res}{\operatorname{Res}}

\newcommand\sgn{\operatorname{sgn}}
\newcommand\ad{\operatorname{ad}}
\newcommand\Hom{\operatorname{Hom}}

\newcommand\id{\operatorname{id}}

\def\pf{\begin{proof}}
\def\epf{\end{proof}}

\newcommand\Alg{\operatorname{Alg}}

\newcommand\cop{{\operatorname{cop}}}

\newcommand\gr{{\operatorname{gr}}}

\renewcommand\o{\otimes}

\newcommand\w{\widetilde}

\begin{document}

%\renewcommand{\baselinestretch}{1.2}

%\thispagestyle{empty}
%\vspace*{2in}
%\title[reductive pointed Hopf algebras]{On reductive pointed Hopf algebras}
\title[Representations of Hopf algebras of dimension 72]{On a family of Hopf algebras of dimension 72}
\author[andruskiewitsch and vay]
{Nicol\'as Andruskiewitsch and Cristian Vay}

\address{FaMAF-CIEM (CONICET), Universidad Nacional de C\'ordoba,
Medina A\-llen\-de s/n, Ciudad Universitaria, 5000 C\' ordoba, Rep\'
ublica Argentina.} \email{andrus@famaf.unc.edu.ar,
vay@famaf.unc.edu.ar}

\thanks{This work was partially supported by ANPCyT-Foncyt, CONICET, Ministerio de Ciencia y
Tecnolog\'{\i}a (C\'ordoba)  and Secyt (UNC)}

\begin{abstract}
We investigate a family of Hopf algebras of dimension 72
whose coradical is isomorphic to the algebra of functions on $\Sn_3$.
We determine the lattice of submodules of the so-called Verma modules and as a consequence we classify all simple modules.
We show that these Hopf algebras are unimodular (as well as their duals) but not quasitriangular; also, they are cocycle deformations of each other.
\end{abstract}

\maketitle

\setcounter{tocdepth}{1}

%\tableofcontents

%\tableofcontents \setcounter{section}{0}
\section*{Introduction}
The study of \fd{} Hopf algebras over an algebraically closed field $\ku$ of characteristic 0
is split into two different classes: the class of semisimple Hopf algebras and the rest.
The Lifting Method from \cite{AS-cambr} is designed to deal with non-semisimple Hopf algebras whose coradical is a Hopf
subalgebra\footnote{An adaptation to general non-semisimple Hopf algebras was recently proposed in \cite{AC}.}.
Pointed Hopf algebras, that is Hopf algebras whose coradical is a group algebra, were intensively studied by this Method.
It is natural to consider next the class of Hopf algebras whose coradical is the algebra $\ku^G$ of functions on a non-abelian group $G$. This class seems to
be interesting at least by the following reasons:

\bigbreak
$\bullet$ The categories of Yetter-Drinfeld modules over the group algebra $\ku G$ and $\ku^G$, $G$ a finite group, are equivalent. Thence, a lot
  sensible information needed for the Lifting Method (description of Yetter-Drinfeld modules, determination of \fd{} Nichols algebras)
  can be translated from the pointed case to this case --or vice versa.

\bigbreak
$\bullet$ The representation theory of Hopf algebras whose coradical is the algebra of functions on a non-abelian group looks easier that the
  the representation theory of pointed Hopf algebras with non-abelian group, because the representation theory of $\ku^G$
  is easier than that of $G$. Indeed, $\ku^G$ is a semisimple abelian algebra and we may
  try to imitate the rich methods in representation theory of Lie algebras, with $\ku^G$ playing the role of the Cartan subalgebra.
  We believe that the representation theory of Hopf algebras with coradical $\ku^G$ might be helpful to study Nichols algebras and deformations.

We have started the consideration of this class in \cite{AV}, where \fd{} Hopf algebras whose coradical is $\ku^{\Sn_3}$
were classified and, in particular, a new family of Hopf algebras of dimension 72 was defined. The purpose of the present paper is to study these Hopf algebras. We first discuss in Section \ref{sect:preliminaries} some general ideas about modules induced from simple $\ku^{G}$- modules,
that we call Verma modules.
We introduce in Section \ref{sect:gral-sn} a new family of Hopf algebras, as a generalization of the construction in \cite{AV},
attached to the class of transpositions in $\Sn_n$ and depending on a parameter $\ba$.

Our main contributions are in Section \ref{sect:modules}:
we determine the lattice of submodules of the various Verma modules and as a consequence we classify all simple modules over
the Hopf algebras of dimension 72 introduced in \cite{AV}. Some further information on these Hopf algebras is given in Section \ref{sec: tipo de rep} and Section \ref{sect:more-info}.

We assume that the reader has some familiarity with Yetter-Drinfeld modules and Nichols algebras $\toba(V)$; we refer to \cite{AS-cambr} for these matters.

\subsection*{Conventions}

\

If $V$ is a vector space, $T(V)$ is the tensor algebra of $V$. If $S$ is a subset of $V$, then we denote by $\langle S\rangle$ the vector subspace generated by $S$. If $A$ is an algebra and $S$ is a subset of $A$, then we denote by $(S)$ the two-sided ideal generated by $S$
and by $\ku\langle S\rangle$ the subalgebra generated by $S$.
If $H$ is a Hopf algebra, then $\Delta$, $\epsilon$, $\mathcal{S}$
denote respectively the comultiplication, the counit and the antipode.
We denote by $\widehat R$ the set of isomorphism classes of a simple $R$-modules, $R$ an algebra; we identify a class in $\widehat{R}$
with a representative without further notice. If $S$, $T$ and $M$ are $R$-modules, we say that \emph{$M$ is an extension of $T$ by $S$} when $M$ fits into an exact sequence $0\rightarrow S\rightarrow M\rightarrow T\rightarrow 0$.

\section{Preliminaries}\label{sect:preliminaries}

\subsection{The induced representation}\label{subsect:induced-generalities}

\

We collect well-known facts about the induced representation. Let $B$ be a subalgebra of an algebra $A$ and let $V$ be a left $B$-module. The induced module is
$\Ind_{B}^{A} V = A \o_{B} V$. The induction has the following  properties:

\begin{itemize}
  \item Universal property:
if $W$ is an $A$-module and $\varphi: V \to W$ is morphism of $B$-modules,
then it extends to a morphism of $A$-modules $\overline{\varphi}: \Ind_{B}^{A} V \to W$.
Hence, there is a natural isomorphism (called Frobenius reciprocity):
$\Hom_{B} (V, \Res _B^A W) \simeq \Hom_{A} (\Ind_{B}^{A} V, W)$.
In categorical terms, \emph{induction is left-adjoint to restriction}.

\medbreak
\item Any \fd{} simple $A$-module is a quotient of the induced module
of a simple $B$-module.
\end{itemize}

Indeed, let $S$ be a \fd{} simple $A$-module and let $T$ be a simple $B$-submodule of $S$. Then the induced
morphism $\Ind_{B}^{A} T \to S$ is surjective.

\begin{itemize}
\item If $B$ is semisimple, then any induced module is projective.
\end{itemize}

The induction functor, being left adjoint to the restriction one, preserves projectives, and any module over a semisimple algebra is projective.

\begin{itemize}
\item If $A$ is a free right $B$-module, say $A\simeq B^{(I)}$, then $\Ind_{B}^{A} V = B^{(I)} \o_{B} V =  V^{(I)}$
as $B$-modules, and a fortiori as vector spaces.
\end{itemize}

\smallbreak
We summarize these basic properties in the setting of \fd{} Hopf algebras, where freeness over Hopf subalgebras is known \cite{NZ}. Also, \fd{} Hopf algebras are Frobenius, so that injective modules are projective and vice versa.

\begin{prop}\label{pr:induced}
Let $A$ be a \fd{} Hopf algebra and let $B$ be a semisimple Hopf subalgebra.
\begin{itemize}
  \item If $T\in \widehat B$, then $\dim \Ind_{B}^{A} T = \dfrac{\dim T\dim A}{\dim B}$.

\smallbreak  \item Any \fd{} simple $A$-module is a quotient of the induced module
of a simple $B$-module.

\smallbreak  \item The induced module of a \fd{} $B$-module is injective and projective.\qed
\end{itemize}
\end{prop}

\subsection{Representation theory of Hopf algebras with coradical a dual group algebra}\label{subsect:U-modules}

\

An optimal situation to apply the Proposition \ref{pr:induced} is when the coradical of the \fd{} Hopf algebra $A$ is a Hopf subalgebra;
in this case $B =$ coradical of $A$ is the best choice. It is tempting to say that the induced module
of a simple $B$-module is a \emph{Verma module} of $A$.

Assume now the coradical $B$ of the \fd{} Hopf algebra $A$ is the algebra of functions $\ku^{G}$ on a finite group $G$.
In this case, we have:

\medbreak
$\bullet$  Any simple $B$-module has dimension 1 and $\widehat{B} \simeq G$; for $g\in G$, the simple module $\ku_g$ has the action $f\cdot 1 = f(g) 1$,
  $f\in\ku^{G}$. Thus any simple $A$-module is a quotient of a  Verma module $M_{g} := \Ind_{\ku^{G}}\ku_g$, for some $g\in G$.

\medbreak
$\bullet$  The ideal $A\delta_g$ is isomorphic to $M_g$ and $A \simeq \oplus_{g\in G} M_g$;
here $\delta_g$ is the characteristic function of the subset $\{g\}$.

\medbreak
$\bullet$\label{bullet:injective-hull}
Let $g\in G$ such that $\delta_g$ is a primitive idempotent of $A$. Since $A$ is Frobenius,  $M_g\simeq A\delta_g$ has a unique simple submodule $S$ and
a unique maximal submodule $N$; $M_g$ is the injective hull of $S$ and the projective cover of $M_g/N$. See \cite[(9.9)]{CR}.

\medbreak
$\bullet$  In all known cases, $\gr A \simeq \toba(V) \# \ku^{G}$, where $V$ belongs to a concrete and short list.
Hence, $\dim M_{g} = \dim \toba(V)$ for any $g\in G$. More than this, in all known cases we dispose of the following information:

\medbreak
\begin{itemize}\renewcommand{\labelitemi}{$\circ$}
  \item There exists a rack $X$ and a 2-cocycle $q\in Z^2(X, \ku^{\times})$
such that $V \simeq (\ku X, c^q)$ as braided vector spaces, see \cite{AG-adv} for details.

\medbreak
  \item\label{item:epimorphism}
  There exists an epimorphism of Hopf algebras $\phi:T(V)\#\ku^{G}\rightarrow A$, see \cite[Subsection 2.5]{AV} for details. Note that $\phi(f\cdot x)=\ad f(\phi(x))$ for all $f\in\ku^{G}$ and $x\in T(V)$.

\medbreak
  \item Let $\X$ be the set of words in $X$, identified with a basis of the tensor algebra $T(V)$. There exists $\B \subset \X$
  such that the classes of the monomials in $\B$ form a basis of $\toba(V)$. The corresponding classes in $A$ multiplied
  with the elements $\delta_g\in \ku^{G}$, $g\in G$, form a basis of $A$.

\medbreak
  \item If $x\in X$, then there exists $g_x\in G$ such that $\delta_h\cdot x =\delta_{h, g_x}x$ for all $h\in G$.
  We extend  this to have $g_x\in G$ for any $x\in \X$.

\medbreak
  \item If $x\in X$, then $x^2 = 0$ in $\toba(V)$ and there exists $f_x\in \ku^{G}$ such that $x^2 = f_x$ in $A$.

  \end{itemize}

Let $g\in G$. If $x\in \B$, then we denote by $m_x$ the class of $x$ in $M_g$. Hence $(m_x)_{x\in \B}$ is a basis of $M_g$.
We may describe the action of $A$ on this basis of $M_g$, at least when we know explicitly the relations of $A$ and the monomials
in $\B$. To start with, let $f\in \ku^{G}$ and $x\in \B$. Then
\begin{equation}\label{eq:action-of-coradical}
\begin{aligned}
f\cdot m_x &= \overline{fx\otimes 1}  = \overline{f\_{1}\cdot x f\_{2}\otimes  1} = \overline{f\_{1}\cdot x \otimes  f\_{2}\cdot  1}\\
& = f(g_xg) \, m_x.
\end{aligned}
\end{equation}

Let now $x= x_1 \dots x_t$ be a monomial in $\B$, with $x_1, \dots, x_t\in X$. Set $y = x_2\dots x_t$;
observe that $y$ need not be in $\B$. Then
\begin{equation}\label{eq:action-of-first-letter}
\begin{aligned}
x_1\cdot m_x &= \overline{x_1^2x_2\dots x_t\otimes 1}  = \overline{f_{x_1}y\otimes  1} = f_{x_1}(g_{y}g) \, \overline{y\otimes  1}.
\end{aligned}
\end{equation}

Let now $M$ be a finite dimensional $A$-module. It is convenient to consider the decomposition of $M$ in isotypic components as $\ku^{G}$-module:
$M = \oplus_{g\in G}M[g]$, where $M[g] = \delta_g\cdot M$. Note that
\begin{align}\label{eq:comp-isotyp}
x\cdot M[g] &= M[g_xg] & \text{for all } x\in \B,\, g&\in G.
\end{align}

For instance, \eqref{eq:action-of-coradical} says that the isotypic components of the Verma module $M_{g}$ are
$M_g[h] = \langle m_x: x\in \B, \, g_xg = h\rangle$.

\section{Hopf algebras related to the class of transpositions in the symmetric group}\label{sect:gral-sn}
%We are interested in a family of Hopf algebras related to the symmetric groups.

\subsection{Quadratic Nichols algebras}\label{subsect:nichols-gral-sn}

\

Let $n\ge 3$; denote by
$\Oc_2^n$ the conjugacy class of $\mbox{\footnotesize (12)}$ in $\Sn_n$ and by $\sgn:C_{\Sn_n}\mbox{\footnotesize (12)} \rightarrow\ku$
the restriction of the sign representation of $\Sn_n$ to the centralizer of $\mbox{\footnotesize (12)}$.
Let $V_n = M(\mbox{\footnotesize (12)},\sgn) \in \ydsnd$\label{V_n};
$V_n$ has a basis $(\xij{ij})_{(ij)\in\Oc_{2}^n}$ such that the action $\cdot$ and the coaction $\delta$
are given by
\begin{align*}
&\delta_h\cdot\xij{ij} =\delta_{h,(ij)}\,\xij{ij}\quad \forall h\in \Sn_n &\mbox{and}&&
\delta(\xij{ij})=\sum_{h\in \Sn_n}\sgn(h)\delta_{h}\ot x_{h^{-1}(ij)h}.
\end{align*}

\bigbreak
Let $n=3, 4, 5$. By  \cite{milinskisch,grania}, we know that $\toba(V_n)$ is quadratic and
\fd; actually, the ideal $\cJ_n$ of relations of $\toba(V_n)$  is generated by
\begin{align}
\label{eq:rels-powers}\xij{ij}^2&,\\
\label{eq:rels-ijkl}R_{(ij)(kl)}&:=\xij{ij}\xij{kl}+\xij{kl}\xij{ij},
\\
\label{eq:rels-ijik}R_{(ij)(ik)}&:=\xij{ij}\xij{ik}+\xij{ik}\xij{jk}+\xij{jk}\xij{ij}
\end{align}
for $(ij),(kl),(ik)\in\Oc_2^n$ with $\#\{i,j,k,l\} = 4$.

For $n\geq 6$, we define the \emph{quadratic Nichols algebra} $\toba_n$ in the same way,
that is as the quotient of the tensor algebra $T(V_n)$ by the ideal generated by the quadratic
relations \eqref{eq:rels-powers}, \eqref{eq:rels-ijkl} and \eqref{eq:rels-ijik}
for $(ij),(kl),(ik)\in\Oc_2^n$ with $\#\{i,j,k,l\} = 4$.
It is however open whether:

\begin{itemize}
    \item $\toba(V_n)$ is quadratic, i. e. isomorphic to $\toba_n$;

    \item the dimension of $\toba(V_n)$ is finite;

    \item the dimension of $\toba_n$ is finite.
\end{itemize}

But we do know that the only
possible finite dimensional Nichols algebras\footnote{There is one exception when $n = 4$ that is \fd{} and two exceptions when $n=5$ and 6 that are not known.}
over $\Sn_n$ are related to the orbit of transpositions and a pair of characters
\cite[Th. 1.1]{AFGV}. Also, the  Nichols algebras related to these two characters are twist-equivalent \cite{ve}.

\subsection{The parameters}\label{subsect:parameters-gral-sn}

\

We consider the set of parameters
$$
\gA_n :=\Big\{ \ba=(\aij{ij})_{(ij)\in\Oc_2^n}\in\ku^{\Oc_2^n}: \sum_{(ij)\in\Oc_2^n}\aij{ij}=0\Big\}.
$$
The group $\Gamma_n := \ku^{\times}\times\Aut(\Sn_n)$ acts on  $\gA_n$ by
\begin{align}\label{equ:action}
(\mu,\theta)\triangleright\ba&=\mu(a_{\theta(ij)}), & \mu&\in \ku^{\times}, & \theta &\in\Aut(\Sn_n),& \ba & \in \gA_n.
\end{align}
Let $[\ba]\in \Gamma_n\backslash\gA_n$ be the class of $\ba$ under this action.
Let $\trid$ denote also the conjugation action of $\sn$ on itself, so that\footnote{It is well-known that
$\Sn_n$ identifies with the group of inner automorphisms and that this equals $\Aut \Sn_n$, except for $n=6$.} $\Sn_n < \{e\}\times \Aut(\Sn_n)<\Gamma_n$.
Let  $\Sn_{n}^{\ba}=\{g\in\Sn_n|g\triangleright\ba=\ba\}$ be the isotropy group  of $\ba$ under the action of $\sn$.

\medbreak We fix $\ba\in \gA_n$ and introduce
\begin{align}
\label{eq:fij-n}
f_{ij} &= \sum_{g\in\Sn_n}(\aij{ij} - a_{g^{-1}(ij)g})\delta_g \in \ku^{\Sn_n}, & (ij)\in\Oc_2^n.
\end{align}
Clearly,
\begin{align}\label{eq:fij-propiedad}
& &\fij{ij}(ts) &=\fij{ij}(s) & \forall &t\in C_{\Sn_n}\mbox{\footnotesize (ij)}, \quad s\in\Sn_n.
\end{align}

\begin{Def}\label{def:linked}
We say that $g$ and $h \in \Sn_n$ are $\ba$-\emph{linked}, denoted $g\link h$, if either $g = h$ or else there exist $(i_mj_m)$, \dots, $(i_{1}j_{1})\in\cO_2^n$ such that
\begin{itemize}
  \item $g = (i_mj_m)\cdots(i_{1}j_{1})h$,
  \item $\fij{i_sj_s}((i_sj_s)(i_{s-1}j_{s-1})\cdots(i_{1}j_{1})h)\neq0$ for all $1\leq s\leq m$.
\end{itemize}
\end{Def}

In particular, $\fij{i_1j_1}(h) \neq 0$ by \eqref{eq:fij-propiedad}.
We claim that $\link$ is an equivalence relation. For,
if $g$ and $h \in \Sn_n$ are $\ba$-linked, then  $h=(i_1j_1)\cdots(i_{m}j_{m})g$ and
\begin{align*}
%\fij{i_{m}j_{m}}(g) \overset{\eqref{eq:fij-propiedad}}= \fij{i_{m}j_{m}}((i_{m}j_{m})g)&=\fij{i_mj_m}((i_{m-1}j_{m-1})\cdots(i_{1}j_{1})h)\neq0,\\
\noalign{\smallbreak}
\fij{i_{s}j_{s}}((i_{s}j_{s})(i_{s+1}j_{s+1})\cdots(i_{m}j_{m})g)&=\fij{i_{s}j_{s}}((i_{s-1}j_{s-1})\cdots(i_{1}j_{1})h)\\
&\overset{\eqref{eq:fij-propiedad}}=\fij{i_{s}j_{s}}((i_{s}j_{s})(i_{s-1}j_{s-1})\cdots(i_{1}j_{1})h)\neq0.
\end{align*}
In the same way, we see that if $g\link h$ and also $h\link z$, then $g\link z$.

\subsection{A family of Hopf algebras}\label{subsect:family-gral-sn}

\

We fix $\ba\in \gA_n$; recall the elements $f_{ij}$ defined in \eqref{eq:fij-n}. Let $\cI_{\ba}$ be the ideal of $T(V_n)\#\ku^{\Sn_n}$ generated by \eqref{eq:rels-ijkl}, \eqref{eq:rels-ijik} and
\begin{align}
\label{eq:rels-powers Aa}
\xij{ij}^2& - f_{ij},
\end{align}
for all $(ij),(kl),(ik)\in\Oc_2^n$ such that $\#\{i,j,k,l\} = 4$.
Then  $$\cA_{[\ba]}:= T(V_n)\#\ku^{\Sn_n}/\cI_{\ba}$$ is a Hopf algebra,  see Remark \ref{obs:Aa-Hopf}. Also, if $\gr\cA_{[\ba]}\simeq\toba(V_n)\#\ku^{\Sn_n}\simeq\gr\cA_{[\bb]}$, then $\cA_{[\ba]} \simeq \cA_{[\bb]}$ if and only if $[\ba]=[\bb]$, what justifies the notation. If $n =3$, then $\gr\cA_{[\ba]}\simeq\toba(V_3)\#\ku^{\Sn_3}$ and $\dim \cA_{[\ba]} = 72$ \cite{AV}; for $n =4,5$ the dimension is finite but we do not know if it is the "right" one;  for $n \geq 6$, the dimension is unknown to be finite.

\begin{Rem}\label{obs:Aa-Hopf}
A straightforward computation shows that
\begin{align*}
\Delta(\xij{ij}^2)&=\xij{ij}^2\ot1+\sum_{h\in\Sn_n}\delta_{h}\ot x_{h^{-1}(ij)h}^2\quad\mbox{ and }\\
\Delta(\fij{ij})&=\fij{ij}\ot1+\sum_{h\in\Sn_n}\delta_{h}\ot f_{h^{-1}(i)h^{-1}(j)}.
\end{align*}
Then $J=\langle\xij{ij}^2 - f_{ij}:(ij)\in\cO_2^n\rangle$ is a coideal. Since $\fij{ij}(e)=0$, we have that $J\subset\ker\epsilon$ and $\mathcal{S}(J)\subseteq \ku^{\Sn_n}J$. Thus $\cI_{\ba} = (J)$ is a Hopf ideal and $\cA_{[\ba]}$ is a Hopf algebra quotient of $T(V_n)\#\ku^{\Sn_n}$.
We shall say that \emph{$\ku^{\Sn_n}$ is a subalgebra of $\cA_{[\ba]}$} to express that the restriction of the projection
$T(V_n)\#\ku^{\Sn_n} \twoheadrightarrow \cA_{[\ba]}$ to $\ku^{\Sn_n}$ is injective.
\end{Rem}

\bigbreak
Let us collect a few general facts on the representation theory of $\cA_{[\ba]}$.

\begin{Rem}\label{obs:fij}
Assume that $\ku^{\Sn_n}$ is a subalgebra of $\cA_{[\ba]}$ and let $M$ be an $\cA_{[\ba]}$-module. Hence
\renewcommand{\theenumi}{\alph{enumi}}   \renewcommand{\labelenumi}{(\theenumi)}
\begin{enumerate}
  \item\label{item:rem-fij-a} If $(ij)\in\Oc_2^n$ satisfies $\fij{ij}(h)\neq0$, then $\rho(\xij{ij}):M[h]\rightarrow M[\mbox{\footnotesize (ij)}h]$
is an isomorphism.

\smallbreak
\item\label{item:rem-fij-b} Let $g\link h \in \Sn_n$.
Then
$\rho(\xij{i_mj_m})\circ\cdots\circ\rho(\xij{i_{1}j_{1}}):M[h]\rightarrow M[g]$
is an isomorphism.
\end{enumerate}
\end{Rem}

\pf  $\rho(\xij{ij}):M[h]\rightarrow M[\mbox{\footnotesize (ij)}h]$ is injective
and $\rho(\xij{ij}):M[\mbox{\footnotesize (ij)}h]\rightarrow M[h]$ is surjective, by  \eqref{eq:rels-powers Aa}. Interchanging the roles of $h$
and $\mbox{\footnotesize (ij)}h$, we get \eqref{item:rem-fij-a}. Now \eqref{item:rem-fij-b} follows from \eqref{item:rem-fij-a}.
\epf

This Remark is particularly useful to compare Verma modules.

\begin{prop}\label{prop: g h linked then the Verma are isomorphic}
Assume that $\dim \cA_{[\ba]} < \infty$ and $\ku^{\Sn_n}$ is a subalgebra of $\cA_{[\ba]}$.
If $g$ and $h$ are $\ba$-linked, then the Verma modules $M_g$ and $M_h$ are isomorphic.
\end{prop}
\begin{proof}
The Verma module $M_h$ is generated by $m_1=1\ot_{\ku^{\Sn_n}}1\in M_{h}[h]$. By Remark \ref{obs:fij} \eqref{item:rem-fij-b}, there exists $m\in M_h[g]$ such that $M_h=\cA_{[\ba]}\cdot m$. Therefore, there is an epimorphism $M_g\twoheadrightarrow M_h$.
Since $\cA_{[\ba]}$ is finite dimensional, all the Verma modules have the same dimension; hence $M_g\simeq M_h$.
\end{proof}

\begin{Def}\label{def:generic}
We say that the parameter $\ba$ is \emph{generic} when any of the following equivalent conditions holds.

\renewcommand{\theenumi}{\alph{enumi}}   \renewcommand{\labelenumi}{(\theenumi)}
\begin{enumerate}
  \item\label{item:def-generic-a} $\aij{ij}\neq\aij{kl}$ for all $(ij) \neq (kl)\in\cO^n_2$.

%\smallbreak
\item\label{item:def-generic-b} $\aij{ij}\neq a_{h\trid (ij)}$ for all $(ij)\in\Oc_2^n$ and all $h\in \Sn_n -
C_{\Sn_n}\mbox{\footnotesize ($ij$)}$.

%\smallbreak
\item\label{item:def-generic-c} $\fij{ij}(h) \neq 0$  for all $(ij)\in\Oc_2^n$ and all $h\in \Sn_n - C_{\Sn_n}\mbox{\footnotesize ($ij$)}$.
\end{enumerate}
\end{Def}

\pf \eqref{item:def-generic-a} $\implies$ \eqref{item:def-generic-b} is clear, since $(ij) \neq h\trid (ij)$ by the assumption on $h$.
\eqref{item:def-generic-b} $\implies$ \eqref{item:def-generic-a} follows since any $(kl) \neq (ij)$ is of the form $(kl) = h\trid (ij)$,
for some $h\notin \sn^{(ij)}$. \eqref{item:def-generic-b} $\iff$ \eqref{item:def-generic-c}: given $(ij)$, we have
$$
\{h\in \Sn_n : \aij{ij} = a_{h\trid (ij)}\} = \{h\in \Sn_n : \fij{ij}(h) = 0\};
$$
hence, one of these sets equals $C_{\Sn_n}\mbox{\footnotesize ($ij$)}$ iff the other does.\epf

\begin{lema}\label{le:bounded in the dimension of modules for aij all different}
Assume that $\ba$ is generic, so that $g\link h$ for all $g,h\in\Sn_n -\{e\}$.
If $\ku^{\Sn_n}$ is a subalgebra of $\cA_{[\ba]}$, then
\renewcommand{\theenumi}{\alph{enumi}}   \renewcommand{\labelenumi}{(\theenumi)}
\begin{enumerate}
\item\label{item:lema-generic-b} If $\cA_{[\ba]}$ is \fd{}, then the Verma modules $M_g$ and $M_h$ are isomorphic,
for all $g,h\in\Sn_n -\{e\}$.

\smallbreak
 \item\label{item:lema-generic-a} If $M$ is an $\cA_{[\ba]}$-module,  then $\dim M[h]=\dim M[g]$ for all $g,h\in\Sn_n -\{e\}$. Thus $\dim M = (n!-1)\dim M[(ij)] + \dim M[e]$.

\smallbreak
\item\label{item:lema-generic-c} If $M$ is simple and $n=3$, then $\dim M[h]\leq1$ for all $h\in\Sn_3 - \{e\}$.

\end{enumerate}

\end{lema}

\begin{proof}  Let $(ij)\in \sn$ and $g\in\Sn_n -\{e\}$.
\begin{itemize}
  \item If $g=(ik)$, then $g\link (ij)$, as $(ik)=(jk)(ij)(jk)$ and $\ba$ is generic.

 \smallbreak \item  If $g=(kl)$ with $\#\{i,j,l,k\}=4$, then $(ij)\link (ik)$ and $(ik)\link (kl)$, hence $(ij)\link (kl)$.

\smallbreak  \item  If $g=(i_1i_2\cdots i_r)$ is an $r$-cycle, then $g=(i_1i_r)(i_1i_2\cdots i_{r-1})$.
Hence  $g\link (ij)$ by induction on $r$.

\smallbreak  \item  Let $g = g_1 \cdots g_{m}$ be the product of the disjoint cycles $g_1, \dots, g_{m}$, with $m\geq 2$;
say $g_1=(i_1\cdots i_r)$, $g_2 = (i_{r+1}\cdots i_{r+s})$ and denote $y = g_3 \cdots g_{m}$.
Then $g=(i_1i_{r+1})(i_1\cdots i_{r+s})y$ and $y\in C_{\Sn_n}\mbox{\footnotesize ($i_1i_{r+1}$)}$. Hence
$g$ and $(ij)$ are linked by induction on $m$.
\end{itemize}

Now \eqref{item:lema-generic-b} follows from Proposition \ref{prop: g h linked then the Verma are isomorphic}
and \eqref{item:lema-generic-a} from Remark \ref{obs:fij}.
If $n =3$ and $M$ is simple, then $\dim\cA_{[\ba]}=72>(\dim M)^2\geq25(\dim M[\mbox{\footnotesize (12)}])^2$ and the last assertion of the lemma follows.
\end{proof}

\medbreak The characterization of all one dimensional $\cA_{[\ba]}$-modules is not difficult.
Let $\thickapprox$ be the equivalence relation  in  $\Oc_2^n$ given by $(ij)\thickapprox(kl)$ iff $\aij{ij}=\aij{kl}$.
Let $\Oc_2^n =\coprod_{s\in \Upsilon} \C_s$ be the associated partition.
If $h\in \Sn_n$, then
\begin{align}\label{eq:conditions-h}
\fij{ij}(h) &= 0 \,\forall (ij)\in\Oc_2^n & &\iff  & h^{-1}\C_sh &= \C_s  \,\forall s\in \Upsilon & &\iff  & h&\in\Sn_n^\ba.
\end{align}

\begin{lema}\label{prop:dim-uno}
Assume that $\ku^{\Sn_n}$ is a subalgebra of $\cA_{[\ba]}$ and let $h\in \Sn_n^\ba$. Then $\ku_h$ is a $\cA_{[\ba]}$-module with the action given by the algebra map $\zeta_h:\cA_{[\ba]}\rightarrow\ku$,
\begin{align}
&\zeta_h(\xij{ij})=0,& (ij)&\in\Oc_2^n &\mbox{ and }& &\zeta_h(f)=f(h),& &f\in\ku^{\Sn_n}.
\end{align}

The one-dimensional representations of $\cA_{[\ba]}$ are all of this form.
\end{lema}

\pf Clearly, $\zeta_h$ satisfies the relations of $T(V_n)\#\ku^{\Sn_n}$,  \eqref{eq:rels-ijkl} and \eqref{eq:rels-ijik};
\eqref{eq:rels-powers Aa} holds because $h$ fulfills \eqref{eq:conditions-h}. Now,
let $M$ be a module of dimension 1. Then $M = M[h]$ for some $h$; thus $\fij{ij}(h)=0$ for all $(ij)\in\Oc_2^n$ by Remark \ref{obs:fij}.
\epf

\section{Simple and Verma modules over Hopf algebras with coradical $\ku^{\Sn_3}$}\label{sect:modules}

\subsection{Verma modules}\label{subsect: cAa1a2}

\

In this Section, we focus on the case $n =3$. Let $\ba\in \gA_3$. Explicitly, $\cA_{[\ba]}$
is the algebra $(T(V_3)\#\ku^{\Sn_3})/\cI_{\ba}$ where $\cI_{\ba}$
is the ideal generated by
\begin{align}\label{eq:rels-ideal}
&R_{(13)(23)}, & &R_{(23)(13)}, & &\xij{ij}^2 - \fij{ij}, & &(ij)\in \Oc_2^3,
\end{align}
where
\begin{align}\label{eq:fij-3}
\begin{aligned}
\fij{13} &= (\aij{13} - \aij{23})(\delta_{(12)}+\delta_{(123)}) + (\aij{13} - \aij{12})(\delta_{(23)}+\delta_{(132)}),
\\\fij{23} &= (\aij{23} - \aij{12})(\delta_{(13)}+\delta_{(123)})+(\aij{23} - \aij{13})(\delta_{(12)}+\delta_{(132)}),\\
\fij{12} &= (\aij{12} - \aij{13})(\delta_{(23)}+\delta_{(123)}) + (\aij{12} - \aij{23})(\delta_{(13)}+\delta_{(132)}).
\end{aligned}
\end{align}

We know from \cite{AV} that $\cA_{[\ba]}$ is a Hopf algebra of dimension 72 and coradical isomorphic to $\ku^{\Sn_3}$,
for any $\ba\in\gA_3$. Furthermore, any \fd{} non-semisimple Hopf algebra with coradical $\ku^{\Sn_3}$ is isomorphic to
$\cA_{[\ba]}$ for some $\ba\in\gA_3$; $\cA_{[\bb]}\simeq\cA_{[\ba]}$ iff $[\ba]=[\bb]$.
Let $\Omega= \fij{13}(\mbox{\footnotesize (12)}\underline{\quad})- \fij{13}$, that is
\begin{equation}\label{eq:omega}
\begin{aligned}
\Omega = &(\aij{23} - \aij{13})(\dij{(12)}-\dij{e})\\ & + (\aij{13} - \aij{12})(\dij{(13)}-\dij{(132)}) + (\aij{12} - \aij{23})(\dij{(23)}-\dij{(123)}).
\end{aligned}
\end{equation}
The following formulae follow from the defining relations:
\begin{align}
\label{eq: rel 12 13 12}
\xij{12}\xij{13}\xij{12}=& \xij{13}\xij{12}\xij{13}+\xij{23}(\aij{13} - \aij{12}),\\
\label{eq: rel 23 12 23}
\xij{23}\xij{12}\xij{23}=& \xij{12}\xij{23}\xij{12}-\xij{13}(\aij{23} - \aij{12})\,\mbox{ and}\\
\label{eq: rel 23 12 13}
\xij{23}\xij{12}\xij{13}=& \xij{13}\xij{12}\xij{23}+\xij{12}\Omega.
\end{align}

Let
$$
\B=\left\{
\begin{matrix}
1, &\xij{13}, &\xij{13}\xij{12}, &\xij{13}\xij{12}\xij{13}, &\xij{13}\xij{12}\xij{23}\xij{12},\\
   &\xij{23}, &\xij{12}\xij{13}, &\xij{12}\xij{23}\xij{12},\\
   &\xij{12}, &\xij{23}\xij{12}, &\xij{13}\xij{12}\xij{23},\\
   &          &\xij{12}\xij{23}
\end{matrix}
\right\}.
$$
Then $\{x\delta_g|x\in \B,\,g\in\Sn_3\}$ is a basis of $\cA_{[\ba]}$ \cite{AV}.
Fix $g\in G$. The classes of
the monomials in $\B$ form a basis of the Verma module $M_g$. Denote by $m_{(ij)\dots (rs)}$ the class of $\xij{ij}\dots\xij{rs}$; we simply
set $\mcuatro = m_{(13)(12)(23)(12)}$.
The action of $\cA_{[\ba]}$ on $M_g$ is described in this basis by the following formulae:
\begin{align}\label{eq:action-Verma-group-uno}
f\cdot m_1 &= f(g) m_1, & f&\in \ku^{\Sn_3};
\\ \label{eq:action-Verma-group-letras} f\cdot m_{(ij)\dots (rs)} &= f(\mbox{\footnotesize (ij)\dots (rs)} g)\, m_{(ij)\dots (rs)},  & f&\in \ku^{\Sn_3};
\\\label{eq:action-Verma-uno} \xij{ij}\cdot m_1 &= \mij{ij}, & (ij)&\in \Oc^{3}_2;
\\ \label{eq:action-Verma-letras1}
\xij{ij}\cdot\mij{ij} &= \fij{ij}(g)m_1, & (ij)&\in \Oc^{3}_2;
\\ \label{eqn:cuentas-muno} \xij{13}\cdot\mij{23}  &= -\mdos{23}{12}-\mdos{12}{13},
\\ \label{eqn:13 en 12} \xij{13}\cdot\mij{12} &= \mdos{13}{12},
\\
\xij{23}\cdot\mij{13} &= -\mdos{12}{23}-\mdos{13}{12},
\\ \label{eqn:23 en 12} \xij{23}\cdot\mij{12} &= \mdos{23}{12},
\\
\xij{12}\cdot\mij{13} &= \mdos{12}{13},
\\  \xij{12}\cdot\mij{23}  &= \mdos{12}{23};
\end{align}
\begin{align}
\xij{13}\cdot\mdos{13}{12} & = \fij{13}(\mbox{\footnotesize (12)}g)\, \mij{12},
\\ \xij{13}\cdot \mdos{12}{13} & = \mtres{13}{12}{13},
\\ \xij{13}\cdot \mdos{23}{12} & =  -\mtres{13}{12}{13} - \fij{13}(\mbox{\footnotesize (23)}g)\, \mij{23}
\\ \xij{13}\cdot \mdos{12}{23} & = \mtres{13}{12}{23};
\\
\xij{23}\cdot\mdos{13}{12} & = -\mtres{12}{23}{12} - \fij{12} (g) \mij{13},
\\ \xij{23}\cdot \mdos{12}{13} & = \mtres{13}{12}{23} + \Omega(g) \mij{12},
\\ \xij{23}\cdot \mdos{23}{12} & = \fij{23}(\mbox{\footnotesize (12)}g)\mij{12},
\\ \xij{23}\cdot \mdos{12}{23} & = \mtres{12}{23}{12}-\mij{13}\fij{23}(\mbox{\footnotesize (13)}),
\\
\xij{12}\cdot\mdos{13}{12} & = \mtres{13}{12}{13}+\mij{23}\fij{13}(\mbox{\footnotesize (23)}),
\\ \xij{12}\cdot \mdos{12}{13} & = \fij{12}(\mbox{\footnotesize (13)}g) \mij{13},
\\ \xij{12}\cdot \mdos{23}{12} & = \mtres{12}{23}{12},
\\ \xij{12}\cdot \mdos{12}{23} & = \fij{12}(\mbox{\footnotesize (23)}g) \mij{23};
\end{align}
\begin{align}
\xij{13}\cdot\mtres{13}{12}{13} & =  \fij{13}(\mbox{\footnotesize (12)}\mbox{\footnotesize (13)}g)\, \mdos{12}{13},
\\ \xij{13}\cdot \mtres{12}{23}{12} & = \mcuatro,
\\ \label{eq:13 act 13 12 23}\xij{13}\cdot \mtres{13}{12}{23} & = \fij{13}(\mbox{\footnotesize (12)}\mbox{\footnotesize (23)}g) \, \mdos{12}{23},
\\
\xij{23}\cdot\mtres{13}{12}{13} & = \mcuatro - (\fij{12}\Omega + (\aij{13} - \aij{12})\fij{23})(g) m_1,
\\ \xij{23}\cdot \mtres{12}{23}{12} & = \fij{12}(g)\mdos{12}{23} + (\aij{12} - \aij{23})\mdos{13}{12},
\\ \label{eq:23 act 13 12 23} \xij{23}\cdot \mtres{13}{12}{23} & = \fij{23}(\mbox{\footnotesize (23)}\mbox{\footnotesize (12)}g)\mdos{12}{13} - \Omega(g)\mdos{23}{12},
\\ \xij{12}\cdot\mtres{13}{12}{13} & =  (\fij{13}(g) + \fij{12}(\mbox{\footnotesize (23)}))\mdos{13}{12}+
\fij{12}(\mbox{\footnotesize (23)})\mdos{12}{23},
\\ \xij{12}\cdot \mtres{12}{23}{12} & = \fij{12}(\mbox{\footnotesize (23)}\mbox{\footnotesize (12)}g)\, \mdos{23}{12},
\\\label{eqn:cuentas-mtres} \xij{12}\cdot \mtres{13}{12}{23} & = - \mcuatro + (\fij{13}(\mbox{\footnotesize (23)})\fij{23} - \fij{12}(\mbox{\footnotesize (13)}\underline{\quad})\fij{13})(g) m_1;
\end{align}
\begin{align}
\label{eq:mcuatro-b} \xij{13}\cdot \mcuatro & = \fij{13}(g)\,  \mtres{12}{23}{12},
\\
\label{eq:mcuatro-a}\xij{23}\cdot \mcuatro & = \fij{23}(g)\mtres{13}{12}{13}+(\fij{13}(\mbox{\footnotesize (23)})\fij{23}+\Omega\fij{12})(g)\mij{23},
\\
\label{eq:mcuatro-c}
\xij{12}\cdot \mcuatro & =-\fij{12}(g)\mtres{13}{12}{23}\\
\notag &+ (\fij{13}(\mbox{\footnotesize (23)})\fij{23}(\mbox{\footnotesize (12)}\underline{\quad})-\fij{12}(\mbox{\footnotesize (23)}\underline{\quad})
\fij{13}(\mbox{\footnotesize (12)}\underline{\quad}))(g)\mij{12};
\end{align}

\bigbreak
To proceed with the description of the simple modules, we split the consideration of the algebras $\cA_{[\ba]}$
into several cases.
%We denote by $\widehat{\cA_{[\ba]}}$ the set of isomorphism classes of simple $\cA_{[\ba]}$-modules.
\begin{itemize}
\medbreak
 \item $\aij{13} = \aij{12} = \aij{23}$. In this case, there is a projection $\cA_{[\ba]}\to \ku^{\Sn_3}$. It is easy to see that any simple
  $\cA_{[\ba]}$-module is obtained from a simple $\ku^{\Sn_3}$-module composing with this projection; thus, $\widehat{\cA_{[\ba]}} \simeq
  \Sn_3$.

\medbreak
  \item $\aij{13} = \aij{12}$ or $\aij{23} = \aij{12}$ or $\aij{13}  = \aij{23}$, but not in the previous case.
  Up to isomorphism, cf. \eqref{equ:action}, we may assume $\aij{12}\neq\aij{13}=\aij{23}$. For shortness, we shall say that
  $\ba$ is \emph{sub-generic}.

\medbreak
  \item $\ba$ is generic.

\end{itemize}

In the next subsections, we investigate these two different cases. Let us consider the decomposition of the Verma module $M_g$
in isotypic components as $\ku^{\Sn_3}$-modules. The isotypic components of the Verma module $M_e$ are
\begin{equation}\label{eq:isotypic}
\begin{aligned}
M_e[e] &= \langle m_1, \mcuatro\rangle, & M_e[\mbox{\footnotesize (12)}] &= \langle \mij{12}, \mtres{13}{12}{23}\rangle, \\
M_e[\mbox{\footnotesize (13)}] &= \langle \mij{13}, \mtres{12}{23}{12} \rangle, &
M_e[\mbox{\footnotesize (23)}] &= \langle \mij{23}, \mtres{13}{12}{13}\rangle, \\
M_e[\mbox{\footnotesize (123)}] &= \langle \mdos{13}{12}, \mdos{12}{23}\rangle,
& M_e[\mbox{\footnotesize (132)}] &= \langle \mdos{12}{13}, \mdos{23}{12}\rangle.
\end{aligned}
\end{equation}
Let $g, h \in \Sn_3$, $(ij)\in \Oc_2^3$. By \eqref{eq:action-Verma-group-letras} and \eqref{eq:comp-isotyp}, we have
\begin{align}\label{eq:isotypic-gral}
M_g[h] &= M_e[hg^{-1}], \\
\label{eq:action-monomials}\xij{ij}\cdot M_g[h] &\subseteq M_g[\mbox{\footnotesize (ij)}h].
\end{align}
It is convenient to introduce the following elements:
\begin{align}\label{eq:msoc}
m_{\textsf{soc}} &= \fij{13}(\mbox{\rm\footnotesize (23)})\fij{23}(\mbox{\rm\footnotesize (13)}) m_1-\mcuatro,
\\\label{eq:mo}
m_{\textsf{o}} &= \mtres{13}{12}{13}+\fij{13}(\mbox{\rm\footnotesize (23)})\mij{23}.
\end{align}

\subsection{Case $\ba\in\gA_3$ generic.}\label{subsec: generic case}

\

To determine the simple $\cA_{[\ba]}$-modules, we just need to determine the maximal submodules of the various Verma modules.
By Lemma \ref{le:bounded in the dimension of modules for aij all different} \eqref{item:lema-generic-b}, we are reduced to consider
the Verma modules $M_e$ and $M_g$ for some fixed $g\neq e$. We choose $g = \mbox{\footnotesize (13)(23)}$; for the sake of an easy exposition, we write
the elements of $\Sn_3$ as products of transpositions.

We start with the following observation.
Let $M$ be a cyclic $\cA_{[\ba]}$-module, generated by $v\in M[\mbox{\footnotesize (13)(23)}]$.
By \eqref{eq:action-monomials} and acting by the  monomials in our basis of $\cA_{[\ba]}$, we see that
$M[\mbox{\footnotesize (23)(13)}]=\langle \xij{13}\xij{23}\cdot v, \xij{23}\xij{12}\cdot v, \xij{12}\xij{13}\cdot v\rangle$.
This weight space is $\neq 0$ by Lemma \ref{le:bounded in the dimension of modules for aij all different} \eqref{item:lema-generic-a},
and a further application of this Lemma gives the following result.

\begin{Rem}\label{obs:ciclico-generico}
Let $M$ be a cyclic $\cA_{[\ba]}$-module, generated by $v\in M[\mbox{\footnotesize (13)(23)}]$.
If $\dim M[\mbox{\footnotesize (23)(13)}]= 1$, then
\begin{equation}\label{eq:module generated by 1 element}
\begin{aligned}
M[\mbox{\footnotesize (23)}]&=\langle\xij{13}\cdot v\rangle,& M[e]&=\langle \xij{12}\xij{23}\cdot v, \xij{13}\xij{12}\cdot v\rangle,\\
M[\mbox{\footnotesize (12)}]&=\langle \xij{23}\cdot v\rangle, & M[\mbox{\footnotesize (13)}]&=\langle\xij{12}\cdot v\rangle, \\
M[\mbox{\footnotesize (13)(23)}]&=\langle v\rangle,
& M[\mbox{\footnotesize (23)(13)}]&=\langle \xij{13}\xij{23}\cdot v\rangle.
\end{aligned}
\end{equation}
\end{Rem}
Thus, any cyclic module as in the Remark has either dimension 5, 6 or 7. Moreover, there is a simple module $L$ like this;
$L$ has a basis $\{v_g|e\neq g\in\Sn_3\}$ and the action is given by
\begin{align}\label{eq:def-L}
v_g&\in L[g], & \xij{ij}\cdot v_g
&=\begin{cases}
v_{(ij)g} & \mbox{ if }\sgn g=1,\\
\fij{ij}(g) v_{(ij)g}& \mbox{ if }\sgn g=-1.\\
\end{cases}
\end{align}
Let $\ku_e$ be as in Lemma \ref{prop:dim-uno}. We shall see that $L$ and $\ku_{e}$ are the only simple modules of $\cA_{[\ba]}$.

\medbreak
The Verma module $M_e$ projects onto the simple submodule $\ku_e$, hence the kernel of this projection is a maximal submodule; explicitly this is
$$N_e=\cA_{[\ba]}\cdot M_e[\mbox{\footnotesize (13)(23)}]=\oplus_{g\link(13)(23)}M_e[g]\oplus\langle\mcuatro\rangle.$$
We see that this is the unique maximal submodule, as consequence of the following description of all submodules of $M_e$.

\begin{lema}\label{le:submodules Me in the generic case}
The submodules of $M_e$ are
\begin{align*}
\langle\mcuatro\rangle\subsetneq\cA_{[\ba]}\cdot v\subsetneq N_e\subsetneq M_e
\end{align*}
for any $v\in M_e[\mbox{\rm\footnotesize (13)(23)}] - 0$. The submodules $\cA_{[\ba]}\cdot v$ and $\cA_{[\ba]}\cdot u$ coincide iff $v\in\langle u\rangle$. The quotients
$\cA_{[\ba]}\cdot v /\langle\mcuatro\rangle$ and  $N_e/ \cA_{[\ba]}\cdot v$ are isomorphic to $L$; and $M_e/N_e$ and $\langle\mcuatro\rangle$ are isomorphic to $\ku_e$.
\end{lema}

\begin{proof}
By \eqref{eq:mcuatro-a}, \eqref{eq:mcuatro-b} and \eqref{eq:mcuatro-c}, we have $\xij{ij}\cdot\mcuatro=0$ for all $(ij)\in\cO_2^3$.
Let
\begin{align*}
v&=\lambda\mdos{23}{12}+\mu\mdos{12}{13} & &\in M_e[\mbox{\footnotesize \mbox{\footnotesize (13)}(23)}] - 0,
\\w &=\mu\mdos{12}{23}+(\mu-\lambda)\mdos{13}{12} & &\in M_e[\mbox{\footnotesize (23)(13)}].
\end{align*}
Using the formulae \eqref{eqn:cuentas-muno} to \eqref{eqn:cuentas-mtres}, we see that $\xij{13}\xij{23}\cdot v$, $\xij{23}\xij{12}\cdot v$ and $\xij{12}\xij{13}\cdot v$ are non-zero multiples of
$w$. That is, $\dim (\cA_{[\ba]}\cdot v)[\mbox{\footnotesize (23)(13)}] = 1$. Also, $\xij{12}\xij{23}\cdot v=-\mu\mcuatro$ and
$\xij{13}\xij{12}\cdot v=\lambda\mcuatro$. Hence
$$\biggl\{v,\,\xij{23}\cdot v,\,\xij{12}\cdot v,\,\xij{13}\cdot v,\, w,\, \mcuatro\biggr\}$$
is a basis of $\cA_{[\ba]}\cdot v$ by Remark \ref{obs:ciclico-generico}.

Let now $N$ be a (proper, non-trivial) submodule of $M_e$. If $N \neq \langle\mcuatro\rangle$, then there exists  $v\in N[\mbox{\footnotesize (13)(23)}] - 0$.
Hence $\cA_{[\ba]}\cdot v$ is a submodule of $N$ and $N[e]=\langle\mcuatro\rangle$ because $m_1\in M_e[e]$ and $\dim M_e[e]=2$. Therefore $N=\cA_{[\ba]}\cdot N[\mbox{\footnotesize (13)(23)}]$.
\end{proof}

It is convenient to introduce the following $\cA_{[\ba]}$-modules which we will use in the Section \ref{sec: tipo de rep}.
\begin{definition}\label{def: Wt ext L by ke - a generic}
Let $\bt\in\gA_3$. We denote by $W_{\bt}(L,\ku_e)$ the $\cA_{[\ba]}$-module with basis $\{w_g:g\in\Sn_3\}$ and action given by
\begin{align*}
&w_g\in W_{\bt}(L,\ku_e)[g],& &\xij{ij}\cdot w_g
=\begin{cases}
0 & \mbox{ if }g=e,\\
w_{(ij)g}   & \mbox{ if }g\neq e\mbox{ and }\sgn g=1,\\
\fij{ij}(g) w_{(ij)g}  & \mbox{ if }g\neq(ij)\mbox{ and }\sgn g=-1,\\
\tij{ij}w_e & \mbox{ if }g=(ij).
\end{cases}
\end{align*}
\end{definition}
The well-definition of $W_{\bt}$ follows from the next lemma.
\begin{lema}\label{le: Wt ext L by ke - a generic}
Let $\bt, \tilde{\bt}\in\gA_3$.
\renewcommand{\theenumi}{\alph{enumi}}   \renewcommand{\labelenumi}{(\theenumi)}
\begin{enumerate}
\item\label{item: Wt ext L by ke - a generic: t=0}
If $\bt=(0,0,0)$, then $W_{\bt}(L,\ku_e)\simeq\ku_e\oplus L$.
\smallbreak
\item \label{item: Wt ext L by ke - a generic: t non 0}
If $\bt\neq(0,0,0)$, then there exists $v\in M_e[\mbox{\rm\footnotesize (13)(23)}] - 0$ such that $W_{\bt}(L,\ku_e)\simeq\cA_{[\ba]}\cdot v$.
\item \label{item: Wt ext L by ke - a generic: t non 0 reciprocal}
If $v\in M_e[\mbox{\rm\footnotesize (13)(23)}] - 0$, then there exists $\bt\neq(0,0,0)$ such that $W_{\bt}(L,\ku_e)\simeq\cA_{[\ba]}\cdot v$.
\smallbreak
\smallbreak
\item \label{item: Wt ext L by ke - a generic: is an ext}
$W_{\bt}(L,\ku_e)$ is an extension of $L$ by $\ku_e$.
\smallbreak
\item \label{item: Wt ext L by ke - a generic: iso}
$W_{\bt}(L,\ku_e)\simeq W_{\tilde{\bt}}(L,\ku_e)$ if and only if
$\bt=\mu\tilde{\bt}$ with $\mu\in\ku^\times$.
\end{enumerate}
\end{lema}

\begin{proof}
\eqref{item: Wt ext L by ke - a generic: t=0} is immediate. If we prove \eqref{item: Wt ext L by ke - a generic: t non 0}, then \eqref{item: Wt ext L by ke - a generic: is an ext} follows from Lemma \ref{le:submodules Me in the generic case}.

\eqref{item: Wt ext L by ke - a generic: t non 0} We set $w_{(13)(23)}=\tij{13}\mdos{23}{12}-\tij{12}\mdos{12}{13}\in M_e[\mbox{\footnotesize \mbox{\footnotesize (13)}(23)}]-0$,
\begin{align*}
&w_{(23)}=\xij{13}\cdot w_{(13)(23)},\quad w_{(13)}=\xij{12}\cdot w_{(13)(23)},\quad w_{(12)}=\xij{23}\cdot w_{(13)(23)},\\
&w_{(23)(13)}=\frac{1}{\fij{23}(\mbox{\footnotesize (13)})}\xij{23}\xij{12}\cdot w_{(13)(23)}\quad\mbox{ and }\quad w_e=\mcuatro.
\end{align*}
By the proof of Lemma \ref{le:submodules Me in the generic case} and \eqref{eq: rel 23 12 23}, we see that $W_{\bt}(L,\ku_e)\simeq\cA_{[\ba]}\cdot w_{(13)(23)}$.
\eqref{item: Wt ext L by ke - a generic: t non 0 reciprocal} follows from the proof of Lemma \ref{le:submodules Me in the generic case}.
\eqref{item: Wt ext L by ke - a generic: iso} Let $\{\tilde{w}_g:g\in\Sn_3\}$ be the basis of $W_{\tilde{\bt}}(L,\ku_e)$ according to Definition \ref{def: Wt ext L by ke - a generic}. Let $F:W_{\bt}(L,\ku_e)\rightarrow W_{\tilde{\bt}}(L,\ku_e)$ be an isomorphism of $\cA_{[\ba]}$-module. Since $F$ is an isomorphism of $\ku^{\Sn_3}$-modules, there exists $\mu_{g}\in\ku^\times$ for all $g\in\Sn_3$ such that $F(w_g)=\mu_g \tilde{w}_g$. In particular, $F$ induces an automorphism of $L$. Since $L$ is simple (cf. Theorem \ref{thm:simples in the generic case}), $\mu_{g}=\mu_L$ for all $g\neq e$. Since
$F(\xij{ij}\cdot w_{(ij)})=\xij{ij}\cdot F(w_{(ij)})$, we see that $\bt=\frac{\mu_L}{\mu_e}\tilde{\bt}$. Conversely, $F$ is well defined for all $\mu_e$ and $\mu_L$ such that $\mu=\frac{\mu_L}{\mu_e}$.
\end{proof}

The Verma module $M_{(13)(23)}$ projects onto the simple module $L$, hence the kernel of this projection is a maximal submodule; explicitly this is
$$N_{(13)(23)} = \cA_{[\ba]}\cdot M_{(13)(23)}[e]=M_{(13)(23)}[e]\oplus\cA_{[\ba]}\cdot m_{\textsf{soc}}.$$
We see that this is the unique maximal submodule, as consequence of the following description of all submodules of $M_{(13)(23)}$.
Recall $m_{\textsf{soc}}$ from \eqref{eq:msoc}.

\begin{lema}\label{le:submodules Mg in the generic case}
The submodules of $M_{(13)(23)}$ are
\begin{align*}
\cA_{[\ba]}\cdot m_{\textsf{soc}} \subsetneq\cA_{[\ba]}\cdot v\subsetneq N_{(13)(23)} \subsetneq M_{(13)(23)}
\end{align*}
for all $v\in M_{(13)(23)}[e]-0$. The submodules $\cA_{[\ba]}\cdot v$ and $\cA_{[\ba]}\cdot u$ coincide iff $v\in\langle u\rangle$. The quotients
$\cA_{[\ba]}\cdot v /\cA_{[\ba]}\cdot m_{\textsf{soc}}$ and  $N_{(13)(23)}/ \cA_{[\ba]}\cdot v$ are isomorphic to $\ku_e$; and $M_{(13)(23)}/N_{(13)(23)}$ and $\cA_{[\ba]}\cdot m_{\textsf{soc}}$ are isomorphic to $L$.
\end{lema}

\begin{proof}
Let $v=\lambda m_1+\mu\mcuatro\in M_{(13)(23)}[\mbox{\footnotesize (13)(23)}] - 0$ and $N = \cA_{[\ba]}\cdot v$.
Using the formulae \eqref{eqn:cuentas-muno} to \eqref{eqn:cuentas-mtres}, we see that
\begin{align*}
\xij{12}\xij{13}\cdot v &=\lambda\mdos{12}{13}-\mu\fij{13}(\mbox{\footnotesize (23)})^2\mdos{23}{12}\,\mbox{ and}\\
\xij{23}\xij{12}\cdot v &=\mu\fij{23}(\mbox{\footnotesize (13)})^2\mdos{12}{13}+\bigl(\lambda+2\mu\fij{13}(\mbox{\footnotesize (23)})
\fij{23}(\mbox{\footnotesize (13)})\bigr)\mdos{23}{12}.
\end{align*}
Thus, $\dim N[\mbox{\footnotesize (23)(13)}]= 1$ iff $\lambda+\mu\fij{13}(\mbox{\footnotesize (23)})
\fij{23}(\mbox{\footnotesize (13)})=0$, that is iff $v \in \langle m_{\textsf{soc}}\rangle - 0$. In this case,
$$\biggl\{v,\,\xij{23}\cdot v,\,\xij{12}\cdot v,\,\xij{13}\cdot v,\, \xij{12}\xij{13}\cdot v\biggr\}$$
is a basis of $\cA_{[\ba]}\cdot m_{\textsf{soc}}$ by Remark \ref{obs:ciclico-generico}.

Let now $N$ be an arbitrary submodule of $M_{(13)(23)}$. If $\dim N[\mbox{\footnotesize (13)(23)}]= 2$, then  $N = M_{(13)(23)}$.
If $\dim N[\mbox{\footnotesize (13)(23)}]= 0$, then $N\subset M_{(13)(23)}[e]$ by Lemma \ref{le:bounded in the dimension of modules for aij all different}.
But this is not possible since $\ker\xij{13}\cap\ker\xij{23}\cap\ker\xij{12}=0$, what is checked using the formulae
\eqref{eqn:cuentas-muno} to \eqref{eq:mcuatro-c}. It remains the case
$\dim N[\mbox{\footnotesize (13)(23)}]=1$.
By the argument at the beginning of the proof, the lemma follows.
\end{proof}

It is convenient to introduce the following $\cA_{[\ba]}$-modules which we will use in the Section \ref{sec: tipo de rep}.
\begin{definition}\label{def: Wt ext ke by L - a generic}
Let $\bt\in\gA_3$. We denote by $W_{\bt}(\ku_e, L)$ the $\cA_{[\ba]}$-module with basis $\{w_g:g\in\Sn_3\}$ and action given by
\begin{align*}
&w_g\in W_{\bt}(\ku_e, L)[g],& &\xij{ij}\cdot w_g
=\begin{cases}
t_{(ij)} w_{(ij)} & \mbox{ if }g=e,\\
\fij{ij}(g) w_{(ij)g}   & \mbox{ if }g\neq e\mbox{ and }\sgn g=1,\\
w_{(ij)g}  & \mbox{ if }\sgn g=-1.
\end{cases}
\end{align*}
\end{definition}
The well-definition of $W_{\bt}(\ku_e, L)$ follows from the next lemma.
\begin{lema}\label{le: Wt ext ke by L - a generic}
Let $\bt, \tilde{\bt}\in\gA_3$.
\renewcommand{\theenumi}{\alph{enumi}}   \renewcommand{\labelenumi}{(\theenumi)}
\begin{enumerate}
\item\label{item: Wt ext ke by L - a generic: t=0}
If $\bt=(0,0,0)$, then $W_{\bt}(\ku_e, L)\simeq L\oplus\ku_e$.
\smallbreak
\item \label{item: Wt ext ke by L - a generic: t non 0}
If $\bt\neq(0,0,0)$, then there exists $v\in M_{(13)(23)}[e] - 0$ such that $W_{\bt}(\ku_e, L)\simeq\cA_{[\ba]}\cdot v$.
\smallbreak
\item \label{item: Wt ext ke by L - a generic: t non 0 reciprocal}
If $v\in M_{(13)(23)}[e] - 0$, then there exists $\bt\neq(0,0,0)$ such that $W_{\bt}(\ku_e, L)\simeq\cA_{[\ba]}\cdot v$.
\smallbreak
\item \label{item: Wt ext ke by L - a generic: is an ext}
$W_{\bt}(\ku_e, L)$ is an extension of $\ku_e$ by $L$.
\smallbreak
\item \label{item: Wt ext ke by L - a generic: iso}
$W_{\bt}(\ku_e, L)\simeq W_{\tilde{\bt}}(\ku_e, L)$ if and only if
$\bt=\mu\tilde{\bt}$ with $\mu\in\ku^\times$.
\end{enumerate}
\end{lema}

\begin{proof}
\eqref{item: Wt ext ke by L - a generic: t=0} is immediate. If we prove \eqref{item: Wt ext ke by L - a generic: t non 0}, then \eqref{item: Wt ext ke by L - a generic: is an ext} follows from Lemma \ref{le:submodules Mg in the generic case}.

\eqref{item: Wt ext ke by L - a generic: t non 0} We set $w_{(13)(23)}=m_{\textsf{soc}}\in M_{(13)(23)}[\mbox{\footnotesize \mbox{\footnotesize (13)}(23)}]$,
\begin{align*}
&w_{(23)}=\frac{\xij{13}\cdot w_{(13)(23)}}{\fij{13}(\mbox{\footnotesize (13)(23)})},\, w_{(13)}=\frac{\xij{12}\cdot w_{(13)(23)}}{\fij{12}(\mbox{\footnotesize (13)(23)})},\quad w_{(12)}=\frac{\xij{23}\cdot w_{(13)(23)}}{\fij{23}(\mbox{\footnotesize (13)(23)})},
\end{align*}
$w_{(23)(13)}=\xij{23}\xij{12}\cdot w_{(13)(23)}$ and $w_e=-\tij{12}\mdos{13}{12}+\tij{13}\mdos{12}{23}\neq0$
Using the formulae \eqref{eqn:cuentas-muno} to \eqref{eqn:cuentas-mtres}, it is not difficult to see that $W_{\bt}(\ku_e, L)\simeq\cA_{[\ba]}\cdot w_{e}$.
\eqref{item: Wt ext ke by L - a generic: t non 0 reciprocal} follows using the formulae \eqref{eqn:cuentas-muno} to \eqref{eqn:cuentas-mtres}.
The proof of \eqref{item: Wt ext ke by L - a generic: iso} is similar to the proof of Lemma \ref{le: Wt ext L by ke - a generic} \eqref{item: Wt ext L by ke - a generic: iso}.
\end{proof}

\bigbreak

\begin{thm}\label{thm:simples in the generic case}
Let $\ba\in\gA_3$ be generic. There are exactly $2$ simple $\cA_{[\ba]}$modules up to isomorphism, namely $\ku_e$ and $L$.
Moreover, $M_e$ is the projective cover, and the injective hull, of $\ku_e$; also, $M_{(13)(23)}$ is the projective cover, and the injective hull, of $L$.
\end{thm}

\begin{proof}
We know that $\ku_e$ and $L$ are the only two simple $\cA_{[\ba]}$-modules up to isomorphism
by Proposition \ref{pr:induced} and Lemmata \ref{le:bounded in the dimension of modules for aij all different} \eqref{item:lema-generic-b},
\ref{le:submodules Me in the generic case} and \ref{le:submodules Mg in the generic case}. Hence, a set of primitive orthogonal idempotents has at most 6 elements \cite[(6.8)]{CR}. Since the $\delta_g$, $g\in \Sn_3$ are orthogonal idempotents, they must be primitive. Therefore
 $M_e$ and $M_{(13)(23)}$ are  the projective covers (and the injective hulls) of $\ku_e$ and $L$, respectively by \cite[(9.9)]{CR}, see page \pageref{bullet:injective-hull}.
\end{proof}

\subsection{Case $\ba\in\gA_3$ sub-generic.}\label{subsec: non generic case}

\

Through this subsection, we suppose that $\aij{12}\neq\aij{13}=\aij{23}$. Then the equivalence classes of $\Sn_3$ by $\link$ are
\begin{align*}
&\{e\}, & &\{(12)\} & &\text{and } \{(13), (23), (13)(23), (23)(13)\}.
\end{align*}
In fact,
\begin{itemize}
\item $e$ and $(12)$ belong to the isotropy group $\Sn_3^{\ba}$.
\smallbreak

\item  $\mbox{\footnotesize (13)}=\mbox{\footnotesize (23)(12)(23)}$ with $\fij{12}(\mbox{\footnotesize (23)})=\aij{12}-\aij{13}\neq0$ and\\ $\fij{23}(\mbox{\footnotesize (12)(23)})=\aij{23}-\aij{12}\neq0$.

\smallbreak

\item $\mbox{\footnotesize (123)} = \mbox{\footnotesize (13)(23)}$ with $\fij{13}(\mbox{\footnotesize (23)})=\aij{13}-\aij{12}\neq0$.

\smallbreak

\item $\mbox{\footnotesize (132)} =\mbox{\footnotesize (23)(13)}$ with $\fij{23}(\mbox{\footnotesize (13)})=\aij{23}-\aij{12}\neq0$.
\end{itemize}

\smallbreak

To determine the simple $\cA_{[\ba]}$-modules, we proceed as in the subsection above; that is, we just need to
determine the maximal submodules of the Verma modules $M_{e}$, $M_{(12)}$ and $M_{(13)(23)}$,
see Proposition \ref{prop: g h linked then the Verma are isomorphic}.

\smallbreak

Let $M$ be a cyclic $\cA_{[\ba]}$-module generated by $v\in M[\mbox{\footnotesize (13)(23)}]$. Here again, we can describe the weight spaces of $M$. By \eqref{eq:action-monomials} and acting by the  monomials in our basis, we see that
$M[\mbox{\footnotesize (23)(13)}]=\langle \xij{13}\xij{23}\cdot v, \xij{23}\xij{12}\cdot v, \xij{12}\xij{13}\cdot v\rangle$.
This weight space is $\neq 0$ by Remark \ref{obs:fij} applied to $(13)(23)\link(23)(13)$,
and a further application of this Remark gives the following result.

\begin{Rem}\label{obs:ciclico-no-generico}
Let $M$ be a cyclic $\cA_{[\ba]}$-module generated by $v\in M[\mbox{\footnotesize (13)(23)}]$.
If $\dim M[\mbox{\footnotesize (23)(13)}]= 1$, then
\begin{equation}\label{eq:module generated by 1 element in the non generic case}
\begin{aligned}
M[e] =& \langle\xij{23}\xij{13}\cdot v, (\xij{12}\xij{23})\cdot v, \xij{13}\xij{12}\cdot v\rangle, & M[\mbox{\footnotesize (13)(23)}]=\langle &v\rangle,\\
M[&\mbox{\footnotesize (12)}]=\langle \xij{23}\cdot v, (\xij{13}\xij{12}\xij{13})\cdot v\rangle, & M[\mbox{\footnotesize (23)}]=\langle\xij{13}\cdot &v\rangle,
\\ M[\mbox{\footnotesize (23)(13)}]&=\langle \xij{12}\xij{13}\cdot v\rangle,
 & M[\mbox{\footnotesize (13)}]=\langle\xij{12}\cdot &v\rangle.
\end{aligned}
\end{equation}
\end{Rem}

There is a simple module $L$ like this; $\{v_{(13)}, v_{(23)}, v_{(13)(23)}, v_{(23)(13)}\}$ is a basis of $L$ and the action is given by
\begin{align}\label{eq:def-Ltilde}
v_g&\in L[g], &
\xij{ij}\cdot v_g
=\begin{cases}
0 & \mbox{ if }g=(ij)\\
m_{(ij)g} & \mbox{ if }g\neq(ij),\,\sgn g=-1,\\
\fij{ij}(g) m_{(ij)g}& \mbox{ if }\sgn g=1.\\
\end{cases}
\end{align}
Let $\ku_{(12)}$ and $\ku_e$ be as in Lemma \ref{prop:dim-uno}. We shall see that $L$, $\ku_{(12)}$ and $\ku_e$ are the only simple modules of $\cA_{[\ba]}$.

\smallbreak
The Verma module $M_e$ projects onto the simple module $\ku_e$, hence the kernel of this projection is a maximal submodule; explicitly this is $$N_e=\cA_{[\ba]}\cdot \left(M_e[\mbox{\footnotesize (13)(23)}]\oplus M_e[\mbox{\footnotesize (12)}]\right)=\oplus_{g\link(13)(23)}M_e[g]\oplus M_e[\mbox{\footnotesize (12)}]\oplus\langle\mcuatro\rangle.$$

We see that this is the unique maximal submodule, as consequence of the following description of all submodules of $M_e$.
\begin{lema}\label{le:submodules Me in the non generic case}
The lattice of (proper, non-trivial) submodules of $M_{e}$ is displayed in \eqref{eq:lattice-Me-nongeneric},
where $v$ and $w$ satisfy $$M_e[\mbox{\rm\footnotesize (13)(23)}]=\langle v, \mdos{23}{12}\rangle, \qquad M_e[\mbox{\rm\footnotesize (12)}]=\hspace{-3pt}\langle w,\mtres{13}{12}{23}\rangle.$$ The submodules $\cA_{[\ba]}\cdot v$ (resp. $\cA_{[\ba]}\cdot w$) and $\cA_{[\ba]}\cdot v_1$ (resp. $\cA_{[\ba]}\cdot w_1$) coincide iff $v\in \langle v_1\rangle$ (resp. $w\in \langle w_1\rangle$). The labels on the arrows indicate the quotient of the module on top by the module on the bottom.
%The simple subquotients are listed in Table \ref{tab:Subquotients-Me-sub-generic}.
\end{lema}

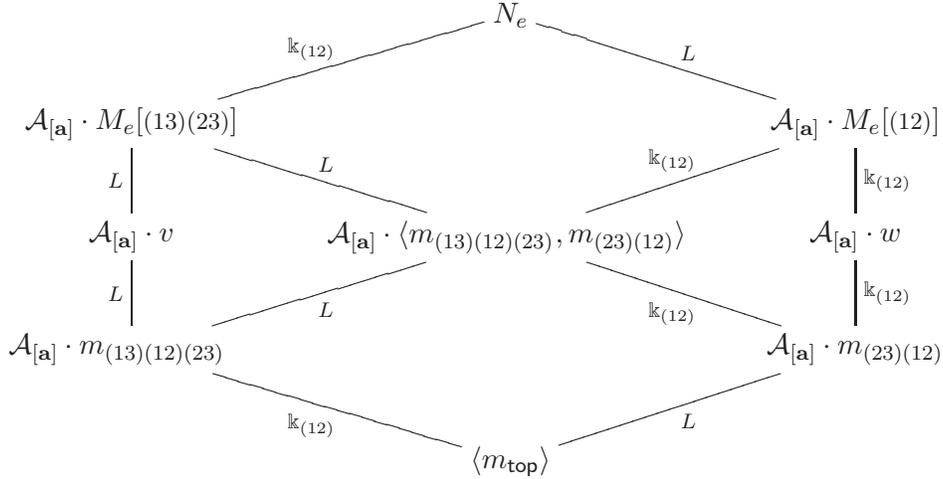
\begin{equation}\label{eq:lattice-Me-nongeneric}
\xymatrix{ & N_{e}  \ar@{-}[1, -1]_{\ku_{(12)}}\ar@{-}[1, 1]^L &  \\
\cA_{[\ba]}\cdot M_{e}[\mbox{\rm\footnotesize (13)(23)}]\ar@{-}[d]_L \ar@{-}[1, 1]^L&
& \cA_{[\ba]}\cdot M_{e}[\mbox{\rm\footnotesize (12)}]
\ar@{-}[d]^{\ku_{(12)}} \ar@{-}[1, -1]_{\ku_{(12)}}\\
\cA_{[\ba]}\cdot v \ar@{-}[d]_L & \ar@{-}[1, -1]^L \hspace{-3pt}\cA_{[\ba]}\cdot\langle\mtres{13}{12}{23},\mdos{23}{12}\rangle
\ar@{-}[1, 1]_{\ku_{(12)}}
& \cA_{[\ba]}\cdot w \ar@{-}[d]^{\ku_{(12)}} \\
\cA_{[\ba]}\cdot \mtres{13}{12}{23} \quad \ar@{-}[1, 1]_{\ku_{(12)}} & & \cA_{[\ba]}\cdot\mdos{23}{12}\ar@{-}[1, -1]^{L}\\
& \langle\mcuatro\rangle & }
\end{equation}

\begin{proof}
Let
\begin{align*}
v&=\lambda\mdos{23}{12}+\mu\mdos{12}{13} & &\in M_e[\mbox{\footnotesize \mbox{\footnotesize (13)}(23)}] - 0,\\
\tilde{v} &=\mu\mdos{12}{23}+(\mu-\lambda)\mdos{13}{12} & &\in M_e[\mbox{\footnotesize (23)(13)}].
\end{align*}
Using the formulae \eqref{eqn:cuentas-muno} to \eqref{eqn:cuentas-mtres}, we see that $\xij{23}\xij{12}\cdot v$ and $\xij{12}\xij{13}\cdot v$ are non-zero multiples of $\tilde{v}$.  That is, $\dim (\cA_{[\ba]}\cdot v)[\mbox{\footnotesize (23)(13)}] = 1$. Moreover,
$\xij{12}\xij{23}\cdot v=-\mu\mcuatro$ and $\xij{13}\xij{12}\cdot v=\lambda\mcuatro$; and
$\xij{23}\cdot v$ and $(\xij{13}\xij{12}\xij{13})\cdot v$ are non-zero multiples of $\mu\mtres{13}{12}{23}$.
By Remark \ref{obs:ciclico-no-generico}, we obtain a basis for $\cA_{[\ba]}\cdot v$:
\begin{align}\label{eq:submodules Me in the non generic case 2}
\biggl\{v,\,\xij{12}\cdot v,\,\xij{13}\cdot v,\, \tilde{v},\, \mcuatro,\,\mu\mtres{13}{12}{23}\biggr\};
\end{align}
if $\mu=0$, we obviate the last vector.

\smallbreak

By \eqref{eq:mcuatro-a}, \eqref{eq:mcuatro-b} and \eqref{eq:mcuatro-c}, $\xij{ij}\cdot\mcuatro=0$ for all $(ij)\in\cO_2^3$. Then $$\cA_{[\ba]}\cdot\mcuatro=\langle\mcuatro\rangle$$
and $\cA_{[\ba]}\cdot u=\cA_{[\ba]}\cdot m_1=M_e$ if $u\in M_e[e]$ is linearly independent to $\mcuatro$.

\smallbreak

By \eqref{eq:13 act 13 12 23}, \eqref{eq:23 act 13 12 23} and \eqref{eqn:cuentas-mtres}, $\xij{ij}\cdot\mtres{13}{12}{23}=-\delta_{(12)}(\mbox{\footnotesize (ij)})\mcuatro$ for all $(ij)\in\cO_2^3$. Then
$$\cA_{[\ba]}\cdot\mtres{13}{12}{23}=\langle\mcuatro, \mtres{13}{12}{23}\rangle.$$
By \eqref{eq:action-Verma-letras1}, \eqref{eqn:13 en 12} and \eqref{eqn:23 en 12}, $\xij{ij}\cdot\mij{12}=\delta_{(13)}(\mbox{\footnotesize(ij)})\mdos{13}{12}+\delta_{(23)}(\mbox{\footnotesize(ij)})\mdos{23}{12}$ for all $(ij)\in\cO_2^3$. Then
$$\cA_{[\ba]}\cdot w=\cA_{[\ba]}\cdot\mdos{23}{12}\oplus\langle w\rangle$$
by \eqref{eq:submodules Me in the non generic case 2} and Remark \ref{obs:fij}, if $w\in M_e[\mbox{\footnotesize (12)}]$ is linearly independent to $\mtres{13}{12}{23}$.

\smallbreak

Let now $N$ be a (proper, non-trivial) submodule of $M_e$ which is not $\langle\mcuatro\rangle$. We set $\widetilde{N} =\cA_{[\ba]}\cdot N[\mbox{\footnotesize (12)}]+ \cA_{[\ba]}\cdot N[\mbox{\footnotesize (13)(23)}]$. Then $\widetilde{N}[g]=N[g]$ for all $g\neq e$ by Remark \ref{obs:fij}. By the argument at the beginning of the proof, $\langle\mcuatro\rangle\subset\widetilde{N}$. Then $\widetilde{N}[e]=\langle\mcuatro\rangle=N[e]$ because otherwise $N=M_e$. Therefore $N=\widetilde{N}$ . To finish, we have to calculate the submodules of $M_e$ generated by homogeneous subspaces of $M_e[\mbox{\footnotesize (12)}]\oplus M_e[\mbox{\footnotesize (13)(23)}]$; this follows from the argument at the beginning of the proof.
\end{proof}

\smallbreak

The Verma module $M_{(13)(23)}$ projects onto the simple module $L$, hence the kernel of this projection is a maximal submodule; explicitly this is
\begin{align*}
N_{(13)(23)} &= \cA_{[\ba]}\cdot\left(M_{(13)(23)}[e]\oplus M_{(13)(23)}[\mbox{\footnotesize (12)}]\right)\\
&=M_{(13)(23)}[e]\oplus M_{(13)(23)}[\mbox{\footnotesize (12)}]\oplus\cA_{[\ba]}\cdot m_{\textsf{soc}}.
\end{align*}

We see that this is the unique maximal submodule, as consequence of the following description of all submodules of $M_{(13)(23)}$.

\begin{lema}\label{le:submodules Mg in the non generic case}
The lattice of (proper, non-trivial) submodules of $M_{(13)(23)}$ is
$$
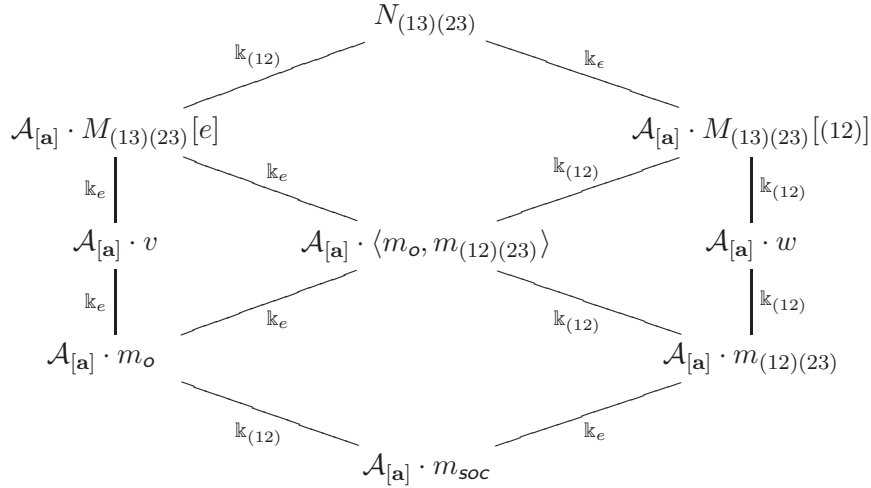
\xymatrix{ & N_{(13)(23)}  \ar@{-}[1, -1]_{\ku_{(12)}} \ar@{-}[1, 1]^{\ku_\epsilon} &  \\
\cA_{[\ba]}\cdot M_{(13)(23)}[e]\ar@{-}[d]_{\ku_e} \ar@{-}[1, 1]^{\ku_e} & &
\cA_{[\ba]}\cdot M_{(13)(23)}[\mbox{\rm\footnotesize (12)}] \ar@{-}[d]^{\ku_{(12)}}\ar@{-}[1, -1]_{\ku_{(12)}}\\
\cA_{[\ba]}\cdot v \ar@{-}[d]_{\ku_e} &
\ar@{-}[1, -1]^{\ku_{e}}\cA_{[\ba]}\cdot\langle m_{\textsf{o}},\mdos{12}{23}\rangle\ar@{-}[1, 1]_{\ku_{(12)}}
& \cA_{[\ba]}\cdot w \ar@{-}[d]^{\ku_{(12)}} \\
\cA_{[\ba]}\cdot m_{\textsf{o}} \quad \ar@{-}[1, 1]_{\ku_{(12)}} & & \cA_{[\ba]}\cdot\mdos{12}{23}\ar@{-}[1, -1]^{\ku_e}\\
& \cA_{[\ba]}\cdot m_{\textsf{soc}} & }
$$
Here $v$ and $w$ satisfy $M_{(13)(23)}[e]=\langle v,\mdos{12}{23}\rangle$,  $M_{(13)(23)}[(12)]=\langle w, m_{\textsf{o}}\rangle$.
The submodules $\cA_{[\ba]}\cdot v$ (resp. $\cA_{[\ba]}\cdot w$) and $\cA_{[\ba]}\cdot v_1$ (resp. $\cA_{[\ba]}\cdot w_1$) coincide iff $v\in \langle v_1\rangle$ (resp. $w\in \langle w_1\rangle$). The labels on the arrows indicate the quotient of the module on top by the module on the bottom.
%The simple subquotients are listed in Table \ref{tab:Subquotients-Mg-sub-generic}.
\end{lema}

\begin{proof}
Let $u=\lambda m_1+\mu\mcuatro\in M_{(13)(23)}[\mbox{\footnotesize (13)(23)}] - 0$.
Using the formulae \eqref{eqn:cuentas-muno} to \eqref{eqn:cuentas-mtres}, we see that
\begin{align*}
\xij{12}\xij{13}\cdot u &=\lambda\mdos{12}{13}-\mu\fij{13}(\mbox{\footnotesize (23)})^2\mdos{23}{12}\,\mbox{ and}\\
\xij{23}\xij{12}\cdot u &=\mu\fij{23}(\mbox{\footnotesize (13)})^2\mdos{12}{13}+\bigl(\lambda+2\mu\fij{13}(\mbox{\footnotesize (23)})
\fij{23}(\mbox{\footnotesize (13)})\bigr)\mdos{23}{12}.
\end{align*}
Thus, $\dim N[\mbox{\footnotesize (23)(13)}]= 1$ iff $\lambda+\mu\fij{13}(\mbox{\footnotesize (23)})
\fij{23}(\mbox{\footnotesize (13)})=0$, that is iff $u\in\langle m_{\textsf{soc}}\rangle-0$. By Remark \ref{obs:ciclico-no-generico},
$$
\cA_{[\ba]}\cdot m_{\textsf{soc}}=\langle m_{\textsf{soc}},\,\xij{12}\cdot m_{\textsf{soc}},\,\xij{13}\cdot m_{\textsf{soc}},\, \xij{12}\xij{13}\cdot m_{\textsf{soc}}\rangle
$$
and $\cA_{[\ba]}\cdot u=\cA_{[\ba]}\cdot m_1 =M_{(13)(23)}$, if $u\in M_{(13)(23)}[\mbox{\footnotesize (13)(23)}]$ is linearly independent to $m_{\textsf{soc}}$.

\smallbreak

By the formulae \eqref{eqn:cuentas-muno} to \eqref{eq:mcuatro-c}, if $u\in\bigl( M_{(13)(23)}[e]\oplus M_{(13)(23)}[\mbox{\footnotesize (12)}]\bigr)-0$, then $0\neq\langle\xij{13}\cdot u, \xij{23}\cdot u\rangle\subset\cA_{[\ba]}\cdot m_{\textsf{soc}}$. Therefore
$$\cA_{[\ba]}\cdot m_{\textsf{soc}}\subset\cA_{[\ba]}\cdot u$$
by Remark \ref{obs:fij}. Also, if $v$ and $w$ satisfy $M_{(13)(23)}[e]=\langle v,\mdos{12}{23}\rangle$ and $M_{(13)(23)}[(12)]=\langle w, m_{\textsf{o}}\rangle$, then
$$
\langle\xij{12}\cdot v\rangle=\langle m_{\textsf{o}}\rangle\quad\mbox{ and }\quad\langle\xij{12}\cdot w\rangle=\langle\mdos{12}{23}\rangle.
$$

\smallbreak

Let now $N$ be a (proper, non-trivial) submodule of $M_{(13)(23)}$ which is not $\cA_{[\ba]}\cdot m_{\textsf{soc}}$. We set $\widetilde{N}= \cA_{[\ba]}\cdot N[e]+\cA_{[\ba]}\cdot N[\mbox{\footnotesize (12)}]$. Then $\widetilde{N}[g]=N[g]$ for $g= e, (12)$ by Remark \ref{obs:fij}.
By the argument at the beginning of the proof, $\cA_{[\ba]}\cdot m_{\textsf{soc}}\subset\widetilde{N}$. Then $\oplus_{g\link(13)(23)}N[g]=\cA_{[\ba]}\cdot m_{\textsf{soc}}=\oplus_{g\link(13)(23)}\tilde{N}[g]$ because otherwise $N=M_{(13)(23)}$. Therefore $N=\widetilde{N}$. To finish, we have to calculate the submodules of $M_{(13)(23)}$ generated by homogeneous subspaces of $M_{(13)(23)}[\mbox{\footnotesize (12)}]\oplus M_{(13)(23)}[e]$; this follows from the argument at the beginning of the proof.
\end{proof}

\smallbreak

The Verma module $M_{(12)}$ projects onto the simple module $\ku_{(12)}$, hence the kernel of this projection is a maximal submodule; explicitly this is
\begin{align*}
N_{(12)} &= \cA_{[\ba]}\cdot\left(M_{(12)}[\mbox{\footnotesize (13)(23)}]\oplus M_{(12)}[e]\right)\\
&=\oplus_{g\link(13)(23)}M_{(12)}[g]\oplus M_{(12)}[e]\oplus\langle\mcuatro\rangle.
\end{align*}

We see that this is the unique maximal submodule, as consequence of the following description of all submodules of $M_{(12)}$.

\begin{lema}\label{le:submodules M12 in the non generic case}
The lattice of (proper, non-trivial) submodules of $M_{(12)}$ is
$$
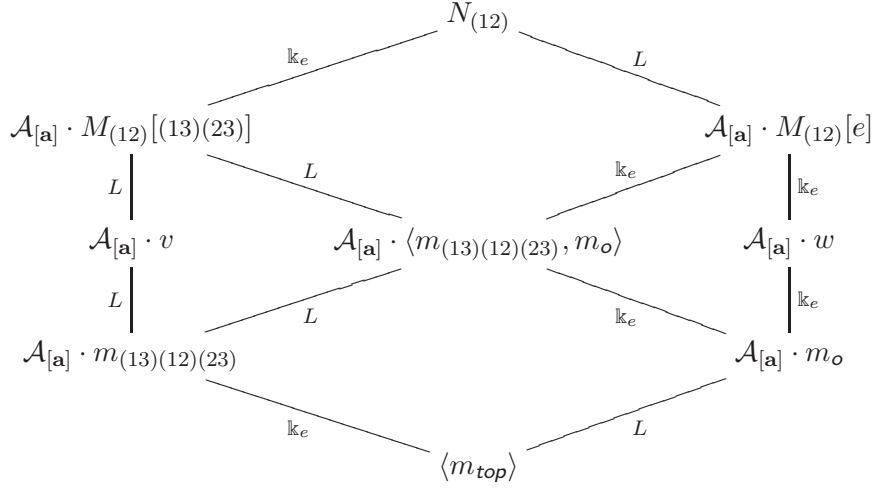
\xymatrix{ & N_{(12)}  \ar@{-}[1, -1]_{\ku_e}\ar@{-}[1, 1]^L &  \\
\cA_{[\ba]}\cdot M_{(12)}[\mbox{\rm\footnotesize (13)(23)}]\ar@{-}[d]_L \ar@{-}[1, 1]^L & &
\cA_{[\ba]}\cdot M_{(12)}[e] \ar@{-}[d]^{\ku_e} \ar@{-}[1, -1]_{\ku_e}\\
\cA_{[\ba]}\cdot v \ar@{-}[d]_L &
\ar@{-}[1, -1]^L \cA_{[\ba]}\cdot\langle\mtres{13}{12}{23}, m_{\textsf{o}}\rangle \ar@{-}[1, 1]_{\ku_e}
& \cA_{[\ba]}\cdot w \ar@{-}[d]^{\ku_e} \\
\cA_{[\ba]}\cdot\mtres{13}{12}{23}\ar@{-}[1, 1]_{\ku_e} & & \cA_{[\ba]}\cdot m_{\textsf{o}}\ar@{-}[1, -1]^L\\
& \langle\mcuatro\rangle & }
$$
Here $v$ and $w$ satisfy $M_{(12)}[\mbox{\rm\footnotesize (13)(23)}]=\langle v, m_{\textsf{o}}\rangle$, $M_{(12)}[e]=\langle w, \mtres{13}{12}{23}\rangle$.
The submodules $\cA_{[\ba]}\cdot v$ (resp. $\cA_{[\ba]}\cdot w$) and $\cA_{[\ba]}\cdot v_1$ (resp. $\cA_{[\ba]}\cdot w_1$) coincide iff $v\in \langle v_1\rangle$ (resp. $w\in \langle w_1\rangle$). The labels on the arrows indicate the quotient of the module on top by the module on the bottom.
%The simple subquotients are listed in Table \ref{tab:Subquotients-M12-sub-generic}.
\end{lema}

\begin{proof}
Let $v=\lambda\mij{23}+\mu\mtres{13}{12}{13}\in M_{(12)}[\mbox{\footnotesize (13)(23)}]$ be a non-zero element. By Remark \ref{obs:ciclico-no-generico} and using the formulae \eqref{eqn:cuentas-muno} to \eqref{eq:mcuatro-c}, we see that
\begin{align}
\notag
(\cA_{[\ba]}\cdot v)[\mbox{\footnotesize (13)(23)}]&=\langle v\rangle,\\
\notag
(\cA_{[\ba]}\cdot v)[\mbox{\footnotesize (13)}]&=\langle(\fij{13}(\mbox{\footnotesize (23)})\mu-\lambda)\mdos{12}{23}
-\mu\fij{13}(\mbox{\footnotesize (23)})\mdos{13}{12}\rangle, \\
\label{eq:submodules M12 in the non generic case 1}
(\cA_{[\ba]}\cdot v)[\mbox{\footnotesize (23)}]&=\langle(\fij{13}(\mbox{\footnotesize (23)})\mu-\lambda)\mdos{12}{13}-\lambda\mdos{23}{12}\rangle,\\
\notag
(\cA_{[\ba]}\cdot v)[\mbox{\footnotesize (23)(13)}]&=\langle(\fij{13}(\mbox{\footnotesize (23)})\mu-\lambda)\fij{23}(\mbox{\footnotesize (13)})\mij{13}+ \lambda\mtres{12}{23}{12}\rangle,\\ \notag
(\cA_{[\ba]}\cdot v)[\mbox{\footnotesize (12)}]&=\langle\mcuatro\rangle\mbox{ and}\\ \notag
(\cA_{[\ba]}\cdot v)[e]&=\langle(\fij{13}(\mbox{\footnotesize (23)})\mu-\lambda)\mtres{13}{12}{23}\rangle.
\end{align}
\smallbreak

By \eqref{eq:mcuatro-a}, \eqref{eq:mcuatro-b} and \eqref{eq:mcuatro-c}, $\xij{ij}\cdot\mcuatro=0$ for all $(ij)\in\cO_2^3$. Then $$\cA_{[\ba]}\cdot\mcuatro=\langle\mcuatro\rangle$$
and $\cA_{[\ba]}\cdot u=\cA_{[\ba]}\cdot m_1=M_e$, if $u\in M_{(12)}[\mbox{\footnotesize (12)}]$ is linearly independent to $\mcuatro$.
By \eqref{eq:13 act 13 12 23}, \eqref{eq:23 act 13 12 23} and \eqref{eqn:cuentas-mtres}, $\xij{ij}\cdot\mtres{13}{12}{23}=-\delta_{(12)}(\mbox{\footnotesize (ij)})\mcuatro$ for all $(ij)\in\cO_2^3$. Then
$$\cA_{[\ba]}\cdot\mtres{13}{12}{23}=\langle\mcuatro, \mtres{13}{12}{23}\rangle.$$
By \eqref{eq:action-Verma-letras1}, \eqref{eqn:13 en 12} and \eqref{eqn:23 en 12}, $\xij{ij}\cdot\mij{12}=\delta_{(13)}(\mbox{\footnotesize(ij)})\mdos{13}{12}+\delta_{(23)}(\mbox{\footnotesize(ij)})\mdos{23}{12}$ for all $(ij)\in\cO_2^3$. Then
$$\cA_{[\ba]}\cdot w=\cA_{[\ba]}\cdot m_{\textsf{o}}\oplus\langle w\rangle$$
by \eqref{eq:submodules M12 in the non generic case 1} and Remark \ref{obs:fij}, if $w\in M_{(12)}[e]$ is linearly independent to $\mtres{13}{12}{23}$.

\bigbreak

Let now $N$ be a (proper, non-trivial) submodule of $M_{(12)}$ which is not $\langle\mcuatro\rangle$. We set $\widetilde{N}=\cA_{[\ba]}\cdot N[e]+ \cA_{[\ba]}\cdot N[\mbox{\footnotesize (13)(23)}]$. Then $\widetilde{N}[g]=N[g]$ for all $g\neq (12)$ by Remark \ref{obs:fij}.
By the argument at the beginning of the proof, $\langle\mcuatro\rangle\subset\widetilde{N}$. Then $N[\mbox{\footnotesize (12)}]=\langle\mcuatro\rangle=\tilde{N}[\mbox{\footnotesize (12)}]$ because otherwise $N=M_{(12)}$. Therefore $N=\widetilde{N}$. To finish, we have to calculate the submodules of $M_{(12)}$ generated by homogeneous subspaces of $M_{(12)}[\mbox{\footnotesize (13)(23)}]\oplus M_{(12)}[e]$; this follows from the argument at the beginning of the proof.
\end{proof}

\bigbreak

As a consequence, we obtain the simples modules in the sub-generic case. The proof of the next theorem runs in the same way as that of Theorem \ref{thm:simples in the generic case}.

\begin{thm}\label{thm:simples in the non generic case}
Let $\ba\in\gA_3$ with $\aij{12}\neq\aij{13}=\aij{23}$. There are exactly $3$ simple $\cA_{[\ba]}$modules up to isomorphism, namely $\ku_e$, $\ku_{(12)}$ and $L$. Moreover, $M_e$ is the projective cover, and the injective hull, of $\ku_e$;
$M_{(12)}$ is the projective cover, and the injective hull, of $\ku_{(12)}$; and $M_{(13)(23)}$ is the projective cover, and the injective hull, of $L$.
\end{thm}

\begin{proof}
We know that $\ku_e$, $\ku_{(12)}$ and $L$ are the only two simple $\cA_{[\ba]}$-modules up to isomorphism
by Proposition \ref{pr:induced} and Lemmata \ref{le:submodules Me in the non generic case}, \ref{le:submodules Mg in the non generic case}
and \ref{le:submodules M12 in the non generic case}.
Hence, a set of primitive orthogonal idempotents has at most 6 elements \cite[(6.8)]{CR}.
Since the $\delta_g$, $g\in \Sn_3$ are orthogonal idempotents, they must be primitive. Therefore
 $M_e$, $M_{(12)}$ and $M_{(13)(23)}$ are respectively the projective covers (and the injective hulls) of $\ku_e$, $\ku_{(12)}$ and $L$ by \cite[(9.9)]{CR}, see page \pageref{bullet:injective-hull}.
\end{proof}

\section{Representation type of $\cA_{[\ba]}$}\label{sec: tipo de rep}

In this section, we assume that $n=3$ as in the preceding one. We will determine the $\cA_{[\ba]}$-modules which are extensions of simple $\cA_{[\ba]}$-modules. As a consequence, we will show that $\cA_{[\ba]}$ is not of finite representation type for all $\ba\in\gA_3$.

\subsection{Extensions of simple modules}\label{subsec: ext of simple mod}

By the following lemma, we are reduced to consider only submodules of the Verma modules for to determine the extensions of simple $\cA_{[\ba]}$-modules. Then we shall split the consideration into three different cases like Section \ref{sect:modules} and use the lemmata there.

\begin{lema}\label{le: extensions are ss or included in verma}
Let $\ba\in\gA_3$ be non-zero. Let $S$ and $T$ be simple $\cA_{[\ba]}$-modules and $M$ be an extension of $T$ by $S$. Hence either $M\simeq S\oplus T$ as $\cA_{[\ba]}$-modules or $M$ is an indecomposable submodule of the Verma module which is the injective hull of $S$.
\end{lema}

\begin{proof}
If there exists a proper submodule $N$ of $M$ which is not $S$, then $M\simeq S\oplus T$ as $\cA_{[\ba]}$-modules. In fact,  $N\cap S$ is either $0$ or $S$ because $S$ is simple. Let $\pi$ be as in \eqref{eq:comm diagram}. Since $T$ is simple, $\pi_{|N}:N\rightarrow T$ results an epimorphism. Therefore $M\simeq S\oplus T$ since $\dim N=\dim(N\cap S)+\dim T$.

Let $M_S$ be the Verma module which is the injective hull of $S$. Then we have the following commutative diagram
\begin{align}\label{eq:comm diagram}
 \xymatrix{
0\ar[r] & S\ar[r]^\imath \ar@{^{(}->}[1,0]& M\ar[r]^\pi \ar@{-->}[1,-1]^f& T\ar[r] & 0 \\
        & M_S            &             &         &
}
\end{align}

Therefore either $M\simeq S\oplus T$ as $\cA_{[\ba]}$-modules or $f$ is inyective. If $f$ is inyective, then $M$ results indecomposable by  Lemmata \ref{le:submodules Me in the generic case} and \ref{le:submodules Mg in the generic case} in the generic case, and by Lemmata \ref{le:submodules Me in the non generic case}, \ref{le:submodules Mg in the non generic case} and \ref{le:submodules M12 in the non generic case} in the sub-generic case.
\end{proof}

\smallbreak

Recall the modules $W_\bt(L,\ku_e)$ and $W_\bt(\ku_e, L)$ from Definitions \ref{def: Wt ext L by ke - a generic} and  \ref{def: Wt ext ke by L - a generic}. The next results follow from Lemmata \ref{le:submodules Me in the generic case}, \ref{le:submodules Mg in the generic case}, \ref{le:submodules Me in the non generic case}, \ref{le:submodules Mg in the non generic case} and \ref{le:submodules M12 in the non generic case} by Lemma \ref{le: extensions are ss or included in verma}.

\begin{lema}\label{le: ext of simple modules a generic}
Let $\ba\in\gA_3$ be generic. Let $S$ and $T$ be simple $\cA_{[\ba]}$-modules and $M$ be an extension of $T$ by $S$.
\renewcommand{\theenumi}{\alph{enumi}}   \renewcommand{\labelenumi}{(\theenumi)}
\begin{enumerate}
\item\label{item: ext of simple modules a generic: S=T}
If $S\simeq T$, then $M\simeq S\oplus S$.
\item\label{item: ext of simple modules a generic: L ke}
If $S\simeq\ku_e$ and $T\simeq L$, then $M\simeq W_{\bt}(L,\ku_e)$ for some $\bt\in\gA_3$.
\item\label{item: ext of simple modules a generic: ke L}
If $S\simeq L$ and $T\simeq\ku_e$, then $M\simeq W_{\bt}(\ku_e, L)$ for some $\bt\in\gA_3$. $\qed$
\end{enumerate}
\end{lema}

\begin{lema}\label{le: ext of simple modules a sub generic}
Let $\ba\in\gA_3$ with $\aij{12}\neq\aij{13}=\aij{23}$. Let $S$ and $T$ be simple $\cA_{[\ba]}$-modules and $M$ be an extension of $T$ by $S$.
\renewcommand{\theenumi}{\alph{enumi}}   \renewcommand{\labelenumi}{(\theenumi)}
\begin{enumerate}
\item\label{item: ext of simple modules a sub generic: S=T}
If $S\simeq T$, then $M\simeq S\oplus S$.
\item\label{item: ext of simple modules a sub generic: k12 ke}
If $S\simeq\ku_e$ and $T\simeq \ku_{(12)}$, then $M\simeq\cA_{[\ba]}\cdot\mtres{13}{12}{23}\subset M_e$.
\item\label{item: ext of simple modules a sub generic: ke k12}
If $S\simeq \ku_{(12)}$ and $T\simeq\ku_e$, then $M\simeq\cA_{[\ba]}\cdot\mtres{13}{12}{23}\subset M_{(12)}$.
\item\label{item: ext of simple modules a sub generic: L ke}
If $S\simeq\ku_e$ and $T\simeq L$, then $M\simeq\cA_{[\ba]}\cdot\mdos{23}{12}\subset M_e$.
\item\label{item: ext of simple modules a sub generic: ke L}
If $S\simeq L$ and $T\simeq\ku_e$, then $M\simeq\cA_{[\ba]}\cdot\mdos{12}{23}\subset M_{(13)(23)}$.
\item\label{item: ext of simple modules a sub generic: L k12}
If $S\simeq \ku_{(12)}$ and $T\simeq L$, then $M\simeq\cA_{[\ba]}\cdot m_{\textsf{o}}\subset M_{(12)}$.
\item\label{item: ext of simple modules a sub generic: k12 L}
If $S\simeq L$ and $T\simeq\ku_{(12)}$, then $M\simeq\cA_{[\ba]}\cdot m_{\textsf{o}}\subset M_{(13)(23)}$. $\qed$
\end{enumerate}
\end{lema}

\smallbreak

\begin{lema}\label{le: extensions of one dimensional simple mod}
Let $\ku_g$ and $\ku_h$ be one-dimensional simple $\cA_{[(0,0,0)]}$-modules and $M$ be an extension of $\ku_h$ by $\ku_g$. Hence
\renewcommand{\theenumi}{\alph{enumi}}   \renewcommand{\labelenumi}{(\theenumi)}
\begin{enumerate}
\item \label{item: extensions of one dimensional simple mod: same sgn} If $\sgn g=\sgn h$, then $M\simeq\ku_g\oplus\ku_h$.
\item \label{item: extensions of one dimensional simple mod: dife sgn} If $\sgn g\neq\sgn h$ and $M$ is not isomorphic to $\ku_g\oplus\ku_h$, then $g=(st)h$ for a unique $(st)\in\cO_3^2$ and $M$ has a basis $\{w_g, w_h\}$ such that $\langle w_g\rangle\simeq \ku_g$ as $\cA_{[\ba]}$-modules, $w_h\in M[h]$ and $\xij{ij} w_h=\delta_{(ij), (st)}w_g$.
\end{enumerate}
\end{lema}

\begin{proof}
$M=M[g]\oplus M[h]$ as $\ku^{\Sn_3}$-modules and $M[g]\simeq\ku_g$ as $\cA_{[\ba]}$-modules. Since $\xij{ij}\cdot M[h]\subset M[\mbox{\footnotesize (ij)}h]$, the lemma follows.
\end{proof}

\subsection{Representation type}
We summarize some facts about the representation type of an algebra.

Let $R$ be an algebra and $\{S_1, ..., S_t\}$ be a complete list of non-isomorphic simple $R$-modules. The \emph{separated quiver of} $R$ is constructed as follows. The set of vertices is $\{S_1, ..., S_t, S_1', ..., S_t'\}$ and we write $\dim\Ext_R^1(S_i, S_j)$ arrows from $S_i$ to $S_j'$, cf. \cite[p. 350]{ARS}. Let us denote by $\Gamma_R$ the underlying graph of the separated quiver of $R$.

A characterization of the hereditary algebras of finite and tame representation type is well-known, see for example \cite{dlab-ringel2}. As a consequence, the next well-known result is obtained. If $R$ is of finite representation type, then it is Theorem D of \cite{dlab-ringel1} or Theorem X.2.6 of \cite{ARS}. The proof given in \cite{ARS} adapts immediately to the case when $R$ is of tame representation type.

\begin{thm}\label{thm: rep type ARS}
Let $R$ be a finite dimensional algebra with radical square zero. Then $R$ is of finite (resp. tame) representation type if and only if $\Gamma_R$ is a finite (resp. affine) disjoint union of Dynkin diagrams. $\qed$
\end{thm}

In order to use the above theorem, we know that

\begin{Rem}\label{obs: for apply thm rep type ARS}
If $\mathfrak{r}$ is the radical of $R$, then the separated quiver of $R$ is equal to the separated quiver of $R/\mathfrak{r^2}$, see for example \cite[Lemma 4.5]{agustin}.
\end{Rem}

We obtain the following result by combining Corollary VI.1.5 and Proposition VI.1.6 of \cite{ARS}.

\begin{prop}\label{prop:combined of ARS}
Let $R$ be an artin algebra, $\chi$ an infinite cardinal and assume there are $\chi$ non-isomorphic indecomposable modules of length $n$. Then $R$ is not of finite representation type. $\qed$
\end{prop}

\smallbreak

Here is the announced result.

\begin{prop}
$\cA_{[(0,0,0)]}$ is of wild representation type. If $\ba\in\gA_3$ is non-zero, then $\cA_{[\ba]}$ is not of finite representation type.
\end{prop}

\begin{proof}
If $\ba\in\gA_3$ is generic, we can apply Proposition \ref{prop:combined of ARS} by Lemma \ref{le: Wt ext L by ke - a generic} and Lemma \ref{le: Wt ext ke by L - a generic}. Hence $\cA_{[\ba]}$ is not of finite representation type for all $\ba\in\gA_3$ generic.

Let $\ba\in\gA_3$ be sub-generic or zero. Then $\dim\Ext_{\cA_{[\ba]}}^1(T,S)=0$ if $S\simeq T$ by Lemma \ref{le: ext of simple modules a sub generic} and \ref{le: extensions of one dimensional simple mod}, and $\dim\Ext_{\cA_{[\ba]}}^1(T,S)=1$ in otherwise. In fact, suppose that $\aij{12}\neq\aij{13}=\aij{23}$, $S\simeq\ku_{e}$ and $T\simeq L$. By Lemma \ref{le:submodules Mg in the non generic case} and Therorem \ref{thm:simples in the non generic case}, $L$ admits a projective resolution of the form
$$
... \longrightarrow P^2\longrightarrow M_e\oplus M_{(12)}\overset{F}{\longrightarrow} M_{(13)(23)}\longrightarrow L\longrightarrow 0,
$$
where $F$ is defined by $F_{|M_e}(m_1)=v$ and $F_{|M_{(12)}}(m_1)=w$; here $v$ and $w$ satisfy $M_{(13)(23)}[e]=\langle v,\mdos{12}{23}\rangle$,  $M_{(13)(23)}[(12)]=\langle w, m_{\textsf{o}}\rangle$. Then
$$
0\longrightarrow\Hom_{\cA_{[\ba]}}(M_{(13)(23)},\ku_e)\overset{\partial_0}{\longrightarrow}\Hom_{\cA_{[\ba]}}(M_e\oplus M_{(12)},\ku_e)
\overset{\partial_1}{\longrightarrow} ...
$$
and $\Ext_{\cA_{[\ba]}}^1(L,\ku_e)=\ker\partial_1/\operatorname{Im}\partial_0$. Since $M_h$ is generated by $m_1\in M_h[h]$ for all $h\in\Sn_3$, $\Hom_{\cA_{[\ba]}}(M_{(13)(23)},\ku_e)=0$ and $\dim\Hom_{\cA_{[\ba]}}(M_e\oplus M_{(12)},\ku_e)=1$. By Lemma \ref{le: ext of simple modules a sub generic}, we know that there exists a non-trivial extension of $L$ by $\ku_e$ and therefore $\dim\Ext_{\cA_{[\ba]}}^1(L,\ku_e)=1$ because it is non-zero. For other $S$ and $T$ and for the case $\ba=(0,0,0)$, the proof is similar.

\smallbreak
Hence if $\ba\in\cA_{[\ba]}$ is sub-generic and $\aij{12}\neq\aij{13}=\aij{23}$, the separated quiver of $\cA_{[\ba]}$ is
$$
\xymatrix{
\ku_e\ar@{->}[d]\ar@{->}[r] & \ku_{(12)}' & L  \ar@//[d] \ar@//[l]
\\
L'& \ku_{(12)}  \ar@{->}[l]\ar@{->}[r]& \ku_e';
}
$$
and the separated quiver of $\cA_{[(0,0,0)]}$ is
$$
\xymatrix{
&\ku_e\ar@{->}[d]\ar@{->}[dr]\ar@{->}[dl]& &
&\ku_{(12)}\ar@{->}[d]\ar@{->}[dr]\ar@{->}[dl]&
\\
\ku_{(12)}'& \ku_{(13)}'& \ku_{(23)}' &
\ku_{e}'& \ku_{(13)(23)}'& \ku_{(23)(13)}'
\\
\ku_{(13)(23)}\ar@{->}[u]\ar@{->}[ur]\ar@{->}[urr]& & \ku_{(23)(13)}\ar@{->}[u]\ar@{->}[ul]\ar@{->}[ull]&
\ku_{(13)}\ar@{->}[u]\ar@{->}[ur]\ar@{->}[urr]& & \ku_{(23)}.\ar@{->}[u]\ar@{->}[ul]\ar@{->}[ull]
}
$$
Therefore the lemma follows from Theorem \ref{thm: rep type ARS} and Remark \ref{obs: for apply thm rep type ARS}.
\end{proof}

\begin{Rem}
Let $\ba\in\gA_3$ be generic. It is not difficult to prove that the separated quiver of $\cA_{[\ba]}$ is
\begin{align*}
&\xymatrix{
\ku_e\ar@{->}[r]\ar@{->}@<-1ex>[r]  & L'} & &\xymatrix{ L\ar@{->}[r]\ar@{->}@<-1ex>[r] & \ku_e'.}
\end{align*}
\end{Rem}

\section{On the structure of $\cA_{[\ba]}$}\label{sect:more-info}

In this section, we assume that $n=3$ as in the preceding one.

\subsection{Cocycle deformations}

\

We show in this subsection that the algebras $\cA_{[\ba]}$ are cocycle deformation of each other.
For this, we first recall the following theorem due to Masuoka.

\smallbreak
If $K$ is a Hopf subalgebra of a Hopf algebra $H$ and $J$ is a Hopf ideal of $K$, then the two-sided ideal $(J)$ of $H$ is in fact a Hopf ideal of $H$.

\begin{thm}\label{thm:cocycle}\cite[Thm. 2]{masuoka}, \cite[Thm. 3.4]{bitidascarainu}.
Suppose that $K$ is Hopf subalgebra of a Hopf algebra $H$. Let $I,J$ be Hopf ideal of $K$. If there is an algebra map $\psi$ from $K$ to $\ku$ such that
\begin{itemize}
 \item $J=\psi\rightharpoonup I\leftharpoonup\psi^{-1}$ and
 \item $H/(\psi\rightharpoonup I)$ is nonzero,
\end{itemize}
then $H/(\psi\rightharpoonup I)$ is
a $(H/(I),H/(J))$-biGalois object and so the quotient Hopf algebras $H/(I)$, $H/(J)$
are monoidally Morita-Takeuchi equivalent. If $H/(I)$ and $H/(J)$ are finite dimensional,
then $H/(I)$ and $H/(J)$  are cocycle deformations of each other. {$\qed$}
\end{thm}

We will need the following lemma to apply the Masuoka's theorem.

\begin{lema}\label{le:tensor subalgebra}
If $W$ is a vector space and $U$ is a vector subspace of $W^{\ot n}$, then the subalgebra of $T(W)$ generated by $U$ is isomorphic to $T(U)$.
\end{lema}

\begin{proof}
It is enough to prove the lemma for $U=W^{\ot n}$. Fix $n$ and let $(x_i)_{i\in I}$ be a basis of $W$.
Then $\mathbf{B}=\{X_{\mathbf{i}}=x_{i_1} \cdots x_{i_n}: \mathbf{i}=(i_1, ..., i_n)\in I^{\times n} \}$ forms a basis of $W^{\ot n}$. Since the $X_{\mathbf{i}}$'s are all homogeneous
elements of the same degree in $T(W)$, we only have to prove that $\{X_{\mathbf{i}_1}\cdots X_{\mathbf{i}_m}:\mathbf{i}_1, ..., \mathbf{i}_m\in I^{\times n}\}$ is linearly independent
in $T(W)$ for all $m\geq1$ and this is true because $\mathbf{B}$ is a basis of monomials of the same degree.
\end{proof}

Here is the announced result. Observe that this gives an alternative proof to the fact that $\dim\cA_{[\ba]} = 72$,
proved in \cite{AV} using the Diamond Lemma.

\begin{prop}\label{prop:cocycle deformations}
For all $\ba\in\gA_3$, $\cA_{[\ba]}$ is a Hopf algebra monoidally Morita-Takeuchi equivalent to $\toba(V_3)\#\ku^{\Sn_3}$.
\end{prop}

\begin{proof} To start with, we consider the algebra $\mathcal{K}_{\ba} := T(V_3)\#\ku^{\Sn_3}/\cJ_{\ba}$, $\ba\in\gA_3$,
where $\cJ_{\ba}$ is the ideal generated by
\begin{align}
\label{eq:rels-powers Ka}R_{(13)(23)},\quad R_{(23)(13)}\quad\mbox{ and }\quad \xij{ij}^2+\sum_{g\in\Sn_3}a_{g^{-1}(ij)g}\,\delta_g,\quad (ij)\in\cO_2^3.
\end{align}

Let $M_3=\ku^{\Sn_3}$ with the regular representation. For all $\ba\in\gA_3$, $M_3$ is an $\mathcal{K}_{\ba}$-module with action given by
\begin{align*}
\quad\xij{ij}\cdot m_g
                                                             =\begin{cases}
                                                             m_{(ij)g} & \mbox{ if }\sgn g=-1,\\
                                                             -a_{g^{-1}(ij)g}\,m_{(ij)g} & \mbox{ if }\sgn g=1.\\
                                                             \end{cases}
\end{align*}
We have to check that the relations defining $\mathcal{K}_{\ba}$ hold in the action. Then
\begin{align*}
\delta_h(\xij{ij}\cdot m_g)&=\delta_h(\lambda_g m_{(ij)g})=\lambda_g\delta_h((ij)g) m_{(ij)g}=\lambda_g\delta_{(ij)h}(g)m_{(ij)g}\\
&=\xij{ij}\cdot(\delta_{(ij)h}\cdot m_{g})
\end{align*}
with $\lambda_g\in\ku$ according to the definition of the action. Note that
$$%\begin{align*}
\xij{ij}\cdot(\xij{ik}\cdot m_g)
=\begin{cases}
-a_{g^{-1}(ik)(ij)(ik)g}\,m_{(ij)(ik)g} & \mbox{ if }\sgn g=-1,\\
-a_{g^{-1}(ik)g}\,m_{(ij)(ik)g} & \mbox{ if }\sgn g=1.\\
\end{cases}
$$%\end{align*}
In any case, we have that $\xij{ij}^2\cdot m_g=-a_{g^{-1}(ij)g}\,m_g$ and
$$
R_{(ij)(ik)}\cdot m_{g}=-(\sum_{(st)\in\cO_2^3}a_{g^{-1}(st)g})m_{(ij)(ik)g}=0.
$$
Let $W=\langle R_{(13)(23)}, \, R_{(23)(13)},\, \xij{ij}^2: (ij)\in\cO_2^3\rangle$ and $K$ be the subalgebra of $T(V_3)$
generated by $W$; $K$ is a braided Hopf subalgebra because $W$ is a Yetter-Drinfeld submodule contained in $\mathcal{P}(T(V_3))$ the primitive elements of $T(V_3)$.
Then $K\#\ku^{\Sn_3}$ is a Hopf subalgebra of $T(V_3)\#\ku^{\Sn_3}$. For each $\ba\in\gA_3$,  by Lemma \ref{le:tensor subalgebra} we can define the algebra morphism $\psi=\psi_K\otimes\epsilon:K\#\ku^{\Sn_3}\rightarrow\ku$ where
$$
\psi_{K|W[g]}=0\,\mbox{ if }\, g\neq e\,\mbox{ and }\, \psi_K(\xij{ij}^2)=-\aij{ij}\,\forall(ij)\in\cO_2^3.
$$

If $J$ denotes the ideal of $K\#\ku^{\Sn_3}$ generated by the generator of $K$, then $\psi^{-1}\rightharpoonup J\leftharpoonup\psi$ is the ideal generated by the generators of $\cI_\ba$. In fact, $\psi^{-1}=\psi\circ\mathcal{S}$ is the inverse element of $\psi$ in the convolution
group $\Alg(K\#\ku^{\Sn_3},\ku)$, $\mathcal{S}(W)[g]\subset(K\#\ku^{\Sn_3})[g^{-1}]$ and $\mathcal{S}(\xij{ij}^2)=-\sum_{h\in\Sn_3}\delta_{h^{-1}} x_{h^{-1}(ij)h}^2$. Then our claim follows if we apply $\psi\ot\id\ot\psi^{-1}$ to $(\Delta\ot\id)\Delta(\xij{ij}^2)=$
$$%\begin{align*}
=\xij{ij}^2\ot1\ot1+\sum_{h\in\Sn_3}\delta_{h}\ot x_{h^{-1}(ij)h}^2\ot1
+\sum_{h,g\in\Sn_3}\delta_{h}\ot\delta_{g}\ot x_{g^{-1}h^{-1}(ij)hg}^2
$$%\end{align*}
and $(\Delta\ot\id)\Delta(x)=x\ot 1\ot1+x_{-1}\ot x_0\ot1+x_{-2}\ot x_{-1}\ot x_0$ for $g\neq e$ and $x\in W[g]$; note that also $x_0\in W[g]$.

The ideal $\psi^{-1}\rightharpoonup J$ is generated by
$$
R_{(13)(23)},\quad R_{(23)(13)}\quad\mbox{ and }\quad\xij{ij}^2+\sum_{g\in\Sn_3}a_{g^{-1}(ij)g}\delta_g\quad \forall(ij)\in\cO_2^3.
$$

Now $\mathcal{K}_\ba = T(V_3)\#\ku^{\Sn_3}/\langle \psi^{-1}\rightharpoonup J\rangle \neq 0$ because it has a non-zero quotient in $\End (M_3)$.
Hence $\cA_{[\ba]}$ is monoidally Morita-Takeuchi equivalent to $\toba(V_3)\#\ku^{\Sn_3}$, by Theorem \ref{thm:cocycle}.
\end{proof}

\subsection{Hopf subalgebras and integrals of $\cA_{[\ba]}$}

\

We collect some information about $\cA_{[\ba]}$. Let
$$\chi=\sum_{g\in\Sn_3}\sgn(g)\delta_g, \quad y = \sum_{(ij)\in\cO_2^3}\xij{ij}.$$ It is easy to see that $\chi$ is a group-like element and that
$y\in \mathcal{P}_{1,\chi}(\cA_{[\ba]})$.

\begin{prop}\label{lema:the unique one dimensional submodule}
Let $\ba\in\gA_3$. Then
\renewcommand{\theenumi}{\alph{enumi}}   \renewcommand{\labelenumi}{(\theenumi)}
\begin{enumerate}
\item\label{item:lema20-a} $G(\cA_{[\ba]})=\{1, \chi\}$.

\smallbreak
\item\label{item:lema20-b} $\mathcal{P}_{1,\chi}(\cA_{[\ba]})=\langle 1 - \chi, y\rangle$.

\smallbreak
\item\label{item:lema20-d} $\ku\langle\chi, y\rangle$ is isomorphic to the 4-dimensional Sweedler Hopf algebra.

\smallbreak
\item\label{item:lema20-i}
The Hopf subalgebras of $\cA_{[\ba]}$ are $\ku^{\Sn_3}$, $\ku\langle \chi\rangle$ and $\ku\langle\chi, y\rangle$.

\smallbreak
\item\label{item:lema20-c} $\mathcal{S}^2(a)=\chi a\chi^{-1}$ for all $a\in\cA_{[\ba]}$.

\smallbreak
\item\label{item:lema20-h} The space of left integrals is $\langle\mcuatro\delta_e\rangle$; $\cA_{[\ba]}$ is unimodular.

\smallbreak
\item\label{item:lema20-g} $(\cA_{[\ba]})^*$ is unimodular.

\smallbreak
\item\label{item:lema20-j} $\cA_{[\ba]}$ is not a quasitriangular Hopf algebra.
\end{enumerate}
\end{prop}

\begin{proof}

We know that the coradical $(\cA_{[\ba]})_0$ of $\cA_{[\ba]}$ is isomorphic to $\ku^{\Sn_3}$ by \cite{AV}.
Since $G(\cA_{[\ba]}) \subset(\cA_{[\ba]})_0$, \eqref{item:lema20-a} follows.

\eqref{item:lema20-b} Recall that $V_3 = M((12),\sgn)\in{}^{\ku^{\Sn_3}}_{\ku^{\Sn_3}}\mathcal{YD}$, see Subsection \ref{subsect:nichols-gral-sn}. Then
$\mathcal{P}_{1,\chi}(\cA_{[\ba]}) / \langle 1 - \chi\rangle$ is isomorphic to the isotypic component of the comodule $V_3$ of type $\chi$.
 That is, if $z=\sum_{(ij)\in\cO_2^3}\lambda_{(ij)} \xij{ij} \in (V_3)_{\chi}$, then
 $$
 \delta (z) = \sum_{h\in G, (ij)\in\cO_2^3}\sgn(h) \lambda_{(ij)}\delta_{h}\ot x_{h^{-1}(ij)h} = \chi\otimes z.
 $$
 Evaluating at $g\otimes \id$ for any $g\in \Sn_3$, we see that $\lambda_{(ij)}=\lambda_{(12)}$ for all $(ij)\in\cO_2^n$. Then $z=\lambda_{(12)}y$.
The proof of \eqref{item:lema20-d} is now evident.

\eqref{item:lema20-i} Let $A$ be a Hopf subalgebra of $\cA_{[\ba]}$. Then $A_{0} = A\cap(\cA_{[\ba]})_0 \subseteq \ku^{\Sn_3}$ by \cite[Lemma 5.2.12]{mongomeri}.
Hence $A_0$ is either $\ku\langle \chi\rangle$ or else $\ku^{\Sn_3}$.
If $A_0=\ku\langle \chi\rangle$, then $A$ is a pointed Hopf algebra with group $\Z/2$.
Hence $A$ is either $\ku\langle \chi\rangle$ or else $\ku\langle\chi, y\rangle$
by \eqref{item:lema20-b} and \cite{N} or \cite{CD}\footnote{The classification of all \fd{} pointed Hopf algebras with group $\Z/2$ also follows easily performing the Lifting method \cite{AS-cambr}.}.
If $A_0 = \ku^{\Sn_3}$, then $A$ is either $\ku^{\Sn_3}$ or else $A = \cA_{[\ba]}$ by \cite{AV}.

To prove \eqref{item:lema20-c}, just note that $\chi\xij{ij}\chi^{-1}=-\xij{ij}$.

\eqref{item:lema20-h} follows from Subsections \ref{subsec: generic case} and \ref{subsec: non generic case}.
Let $\Lambda$ be a non-zero left integral of $\cA_{[\ba]}$. By Lemma \ref{prop:dim-uno}, the distinguished group-like element of $(\cA_{[\ba]})^*$ is $\zeta_h$ for some $h\in\Sn_3^{\ba}$, hence $\Lambda\delta_h=\zeta_h(\delta_h)\Lambda=\Lambda$. Let us consider $\cA_{[\ba]}$ as a left $\ku^{\Sn_3}$-module via the left adjoint action, see page \pageref{item:epimorphism}. Let $\Lambda_g\in(\cA_{[\ba]})[g]$ such that $\Lambda=\sum_{g\in\Sn_3}\Lambda_g$. Then $\Lambda=\delta_e\Lambda=\sum_{s,t\in\Sn_3}\ad\delta_s(\Lambda_{t})\delta_{s^{-1}}\delta_h=\Lambda_{h^{-1}}\delta_h$. Since $M_h\simeq\cA_{[\ba]}\delta_h$, we can use the lemmata of the Section \ref{sect:modules} to compute $\Lambda$.

If $\ba$ is generic, then $h=e$ by Theorem \ref{thm:simples in the generic case}. Since $\xij{ij}\Lambda=0$ for all $(ij)\in\Sn_3$, $\Lambda=\mcuatro\delta_e$ by Lemma \ref{le:submodules Me in the generic case}.

If $\ba$ is sub-generic, we assume that $\aij{12}\neq\aij{13}=\aij{23}$, then either $\Lambda=\Lambda_e\delta_e$ or $\Lambda_{(12)}\delta_{(12)}$ by Theorem \ref{thm:simples in the non generic case}. Since $\xij{ij}\Lambda=0$ for all $(ij)\in\Sn_3$, $\Lambda=\mcuatro\delta_e$ by Lemma \ref{le:submodules Me in the non generic case} and Lemma \ref{le:submodules M12 in the non generic case}.

\eqref{item:lema20-g} By \eqref{item:lema20-c}, $\mathcal{S}^4= \id$. By Radford's  formula for the antipode and \eqref{item:lema20-h},
the distinguished group-like element of $\cA_{[\ba]}$ is central, hence trivial. Therefore, $(\cA_{[\ba]})^*$ is unimodular.

\eqref{item:lema20-j} If there exists $R\in\cA_{[\ba]}\ot\cA_{[\ba]}$ such that $(\cA_{[\ba]}, R)$ is a quasitriangular Hopf algebra, then   $(\cA_{[\ba]}, R)$ has a unique minimal subquasitriangular Hopf algebra $(A_R, R)$ by \cite{radford}. We shall show that such a Hopf subalgebra does not exist using \eqref{item:lema20-i} and therefore $\cA_{[\ba]}$ is not a quasitriangular Hopf algebra.

\smallbreak

By \cite[Prop. 2, Thm. 1]{radford} we know that there exist Hopf subalgebras $H$ and $B$ of $\cA_{[\ba]}$ such that $A_R=HB$ and an isomorphism of Hopf algebras $H^{*\cop}\rightarrow B$. Then $A_R\neq\cA_{[\ba]}$. In fact, let $M(d,\ku)$ denote the matrix algebra over $\ku$ of dimension $d^2$. Then the coradical of $(\cA_{[\ba]})^{*}$ is isomorphic to
\begin{itemize}
 \item $\ku^6$ if $\ba=(0,0,0)$.
 \item $\ku\oplus M(5,\ku)^*$ if $\ba$ is generic by Theorem \ref{thm:simples in the generic case}.
 \item $\ku^2\oplus M(4,\ku)^*$ if $\ba$ is sub-generic by Theorem \ref{thm:simples in the non generic case}.
\end{itemize}
Since $(\cA_{[\ba]})_0\simeq\ku^{\Sn_3}$, $\cA_{[\ba]}$ is not isomorphic to $(\cA_{[\ba]})^{*\cop}$ for all $\ba\in\gA_3$.
Clearly, $A_R$ cannot be $\ku^{\Sn_3}$. Since $\cA_{[\ba]}$ is not cocommutative, $R$ cannot be $1\ot1$. The quasitriangular structures on $\ku\langle\chi\rangle$ and $\ku\langle\chi, y\rangle$ are well known, see for example \cite{radford}. Then it remains the case $A_R\subseteq\ku\langle\chi, y\rangle$ with $R=R_0+R_{\alpha}$ where
$R_0=\frac{1}{2}(1\ot1+1\ot\chi+\chi\ot1-\chi\ot\chi)$ and $R_{\alpha}=\frac{\alpha}{2}(y\ot y+y\ot\chi y+\chi y\ot\chi y-\chi y\ot y)$
for some $\alpha\in\ku$. Since $\Delta(\delta_g)^{\cop}R=R\Delta(\delta_g)$ for all $g\in\Sn_3$, then
\begin{align*}
\Delta(\delta_g)^{\cop}R_0&=R_0\Delta(\delta_g)=\Delta(\delta_g)R_0\quad\mbox{in $\ku^{\Sn_3}$;}
\end{align*}
but this is not possible because $R_0^2=1\ot1$ and $\ku^{\Sn_3}$ is not cocommutative.
\end{proof}

\subsection*{Acknowledgements}
Part of the work of C. V. was done as a fellow of the Erasmus Mundus programme of the EU in the University of Antwerp. He thanks to Prof. Fred Van Oystaeyen for his warm hospitality and help.

\end{document}